\documentclass[a4paper]{amsart}                

\usepackage[latin1]{inputenc}                  
\usepackage[T1]{fontenc}                       
\usepackage[spanish,english]{babel}            
\usepackage{graphicx}                  
\usepackage{amsmath,amssymb,amsthm}            
\usepackage{latexsym}                          
\usepackage{delarray}                          
\usepackage{bbm}                               
\usepackage{hyperref}                          
\usepackage[pdftex,usenames,dvipsnames]{color} 
\usepackage{calrsfs,eufrak,mathrsfs,upgreek}   
\usepackage{bbding,trfsigns}                   
\usepackage{wasysym}                           
\usepackage{datetime}                          
\usepackage{color}                             

\DeclareMathAlphabet{\mathpzc}{OT1}{pzc}{m}{it} 

\newtheorem{teo}{Theorem}
\newtheorem{lema}{Lemma}[section]

\newtheorem{ThA}{Theorem}

\newcommand{\B}{\mathbb{B}}
\newcommand{\C}{\mathbb{C}}
\newcommand{\E}{\mathbb{E}}
\newcommand{\N}{\mathbb{N}}

\newcommand{\R}{\mathbb{R}}

\newcommand{\eps}{\varepsilon}

\DeclareMathAlphabet{\mathpzc}{OT1}{pzc}{m}{it}

\newcommand{\Rad}{\gamma\left(H,\B\right)}

\DeclareFontFamily{U}{mathx}{\hyphenchar\font45}
\DeclareFontShape{U}{mathx}{m}{n}{
      <5> <6> <7> <8> <9> <10>
      <10.95> <12> <14.4> <17.28> <20.74> <24.88>
      mathx10
      }{}
\DeclareSymbolFont{mathx}{U}{mathx}{m}{n}
\DeclareFontSubstitution{U}{mathx}{m}{n}
\DeclareMathAccent{\widecheck}{0}{mathx}{"71}
\DeclareMathAccent{\wideparen}{0}{mathx}{"75}

\allowdisplaybreaks

\author[J.J. Betancor]{Jorge J. Betancor}
\author[A.J. Castro]{Alejandro J. Castro}
\author[J.C. Fari\~na]{Juan C. Fari\~na}
\author[L. Rodr\'iguez-Mesa]{L. Rodr\'iguez-Mesa}

\address{\newline
        Jorge J. Betancor, Alejandro J. Castro, Juan C. Fari\~na and Lourdes Rodr\'iguez-Mesa \newline
        Departamento de An\'alisis Matem\'atico,
        Universidad de La Laguna, \newline
        Campus de Anchieta, Avda. Astrof\'{\i}sico Francisco S\'anchez, s/n, \newline
        38271, La Laguna (Sta. Cruz de Tenerife), Spain}
\email{jbetanco@ull.es, ajcastro@ull.es, jcfarina@ull.es, lrguez@ull.es}

\thanks{This paper is partially supported by MTM2010/17974.}

\date{\today}

\begin{document}

\title[Conical square functions]
{Conical square functions associated with Bessel, Laguerre and Schr\"odinger operators in UMD Banach spaces}

\subjclass[2010]{42B25, 46E40 (primary), 42B20, 46B20, (secondary)}

\keywords{Conical square functions, vector valued harmonic analysis, UMD Banach spaces, Bessel, Laguerre,
Schr\"odinger, Hermite.}

\begin{abstract}
    In this paper we consider  conical square functions in the Bessel, Laguerre and Schr\"odinger settings where the functions
    take values in UMD Banach spaces. Following a recent paper of Hyt\"onen, van Neerven and Portal \cite{HNP},
    in order to define our conical
    square functions, we use $\gamma$-radonifying operators. We obtain new equivalent norms in the Lebesgue-Bochner spaces
    $L^p((0,\infty),\B)$ and $L^p(\R^n,\B)$, $1<p<\infty$, in terms of our square functions, provided that $\B$ is a UMD Banach space.
    Our results can be seen as Banach valued versions of known scalar results for square functions.
\end{abstract}

\maketitle

\section{Introduction}\label{sec:intro}

 In this paper we obtain equivalent norms in the Lebesgue-Bochner space $L^p(\R^n,\B)$, $1<p<\infty$,
 where $\B$ is a UMD Banach space. In order to do this we consider conical square functions defined via fractional
 derivatives of Poisson semigroups associated with Bessel, Laguerre and Schr\"odinger operators. According to
 the ideas developed by Hyt\"onen, van Neerven and Portal \cite{HNP} we use appropriate tent spaces using $\gamma$-radonifying operators
 (or, in other words, methods of stochastic analysis in a Banach valued setting).

 We denote by $P_t(z)$, the classical Poisson kernel in $\R^n$, that is,
 \begin{equation*}\label{eqA}
    P_t(z)
        = c_n\frac{t}{(|z|^2+t^2)^{(n+1)/2}}, \;\;t>0\;\;\mbox{and}\;z\in \R^n,
 \end{equation*}
 where $c_n=\Gamma((n+1)/2)/\pi^{(n+1)/2}$.

 Segovia and Wheeden \cite{SW} introduced fractional derivatives as follows.
 Suppose that $\beta >0$ and $m \in \N$ is such that $m-1 \leq \beta < m$.
 If $F:\Omega \times (0,\infty) \longrightarrow \mathbb C$ is a reasonable nice function,
 where $\Omega \subset \R^n$, the $\beta$-th derivative with respect to $t$ of $F$ is defined by
 $$\partial_t^\beta F(x,t)
    = \frac{e^{-i\pi(m-\beta)}}{\Gamma(m-\beta)}\int_0^\infty \partial_t^m F(x,t+s) s^{m-\beta-1}ds\quad t>0\;\;\mbox{and}\; x\in \Omega.$$
 In \cite{SW} this fractional derivative was used to get characterizations of classical Sobolev spaces.

 As in \cite{TZ} we define the $\beta$-conical square function $S_\beta$ by
 $$S_\beta(f)(x)
    =\left(\int_{\Gamma(x)} \left|t^\beta\partial_t^\beta P_t(f)(y)\right|^2 \frac{dydt}{t^{n+1}}\right)^{1/2}, \quad x \in \R^n,$$
 where $P_t(f)$ denotes the Poisson integral of $f$, that is,
\begin{equation}\label{eq:Poissinteg}
    P_t(f)(x)
        =\int_{\R^n}P_t(x-y)f(y)dy, \quad x\in \R^n,t>0,
\end{equation}
 and, for every $x\in \R^n$, $\Gamma(x)=\{(y,t) \in \R^n\times(0,\infty): |x-y|<t\}$.
 According to \cite[Theorems 5.3 and 5.4]{TZ} the square function $S_\beta$ defines an equivalent
  norm in $L^p(\R^n)$, $1<p<\infty$.

\begin{ThA}\label{ThA}
    Let $1<p<\infty$ and $\beta >0$. Then, there exists $C>0$ such that
    \begin{equation}\label{1.1}
        \frac{1}{C}\|f\|_{L^p(\R^n)}
            \leq \|S_\beta(f)\|_{L^p(\R^n)}
            \leq C\|f\|_{L^p(\R^n)}, \quad f \in L^p(\R^n).
    \end{equation}
\end{ThA}

The equivalence in Theorem~\ref{ThA} for $\beta \in \mathbb N$ can  also be encountered in \cite{HNP},
\cite{MTX} and \cite{Ste1}.

Coifman, Meyer and Stein \cite{CMS} introduced a family of spaces called tent spaces.
These tent spaces are well adapted to certain questions related to harmonic analysis.
Suppose that $1 \leq p,q < \infty$. The tent space $T_p^q(\R^n)$ consists of
all those measurable functions $g$ on $\R^n \times (0,\infty)$ such that
$A_q(g)\in L^p(\R^n)$, where
$$A_q(g)(x)
    =\left(\int_{\Gamma(x)}|g(y,t)|^q\frac{dydt}{t^{n+1}}\right)^{1/q}, \quad x \in \R^n.$$
The norm $\|\cdot\|_{T_p^q(\R^n)}$ in $T_p^q(\R^n)$ is defined by
$\|g\|_{T_p^q(\R^n)}=\|A_q(g)\|_{L^p(\R^n)}$, $g\in T_p^q(\R^n)$.

More  recently Harboure, Torrea and Viviani \cite{HTV} have simplified some proofs of properties
in \cite{CMS} by using vector valued harmonic analysis techniques. Note that the result in Theorem~\ref{ThA}
can be rewritten in terms of tent spaces as follows. If $1<p<\infty$ and $\beta>0$, then, for every
$f \in L^p(\R^n)$, $t^\beta\partial_t^\beta P_t(f) \in T_p^2(\R^n)$ and
$$\frac{1}{C}\|f\|_{L^p(\R^n)}
    \leq \|t^\beta\partial_t^\beta P_t(f)\|_{T_p^2(\R^n)}
    \leq C\|f\|_{L^p(\R^n)},$$
where $C>0$ does not depend on $f$.

Assume that $\B$ is a Banach space. In order to show a version of
Theorem~\ref{ThA} for the Lebesgue-Bochner space $L^p(\R^n,\B)$,
the most natural definition of the $\beta$-conical square
function $S_{\beta,\B}$ is the following
$$S_{\beta,\B}(f)(x)
    =\left(\int_{\Gamma(x)}\|t^\beta\partial_t^\beta P_t(f)(y)\|^2_\B \frac{dydt}{t^{n+1}}\right)^{1/2},
        \quad f\in L^p(\R^n,\B), \ 1<p<\infty.$$
This type of Banach valued conical function has been considered in 
\cite{MTX}, \cite{TZ} and \cite{Xu}. Since a Banach space $\B$ has
Lusin type 2 and Lusin cotype 2 if, and only if, $\B$ is isomorphic
to a Hilbert space, from \cite[Theorems 5.3 and 5.4]{TZ} we can
deduce the following result.

\begin{ThA}\label{ThB}
     Assume that $\B$ is a Banach space, $1<p<\infty$ and $\beta >0$. The following assertions are equivalent:
    \begin{itemize}
        \item[$(i)$] $\B$ is isomorphic to a Hilbert space.
        \item[$(ii)$] There exists $C>0$ such that
        $$\frac{1}{C}\|f\|_{L^p(\R^n,\B)}
            \leq \|S_{\beta,\B}(f)\|_{L^p(\R^n)}
            \leq C\|f\|_{L^p(\R^n,\B)}, \quad f\in L^p(\R^n,\B).$$
    \end{itemize}
\end{ThA}

In order to extend the equivalence (\ref{1.1}) to $L^p(\R^n,\B)$, $1<p<\infty$, when $\B$ is a Banach space which is not
isomorphic to a Hilbert space, Hyt\"onen, van Neerven and Portal
\cite{HNP} introduced new Banach valued tent function spaces. They
considered UMD Banach spaces and $\gamma$-radonifying operators.

As it is well-known, the Hilbert transform $\mathfrak{H}$ defined by
$$\mathfrak{H}(f)(x)
    = \lim_{\eps \rightarrow 0^+} \frac{1}{\pi} \int_{|x-y|>\eps} \frac{f(y)}{x-y}dy, \quad a.e.\;\;x\in \mathbb R,$$
is a bounded operator from $L^p(\mathbb R)$ into itself, $1<p<\infty$,
and from $L^1(\mathbb R)$ into $L^{1,\infty}(\mathbb R)$.

Suppose that $\B$ is a Banach space. The Hilbert transform can be
defined in $L^p(\mathbb R)\otimes \B$, $1<p<\infty$, in the
obvious way. It is said that $\B$ is UMD, provided that
$$\|\mathfrak{H}(f)\|_{L^p(\mathbb R,\B)}
    \leq C_p\|f\|_{L^p(\mathbb R,\B)}, \quad f\in L^p(\mathbb R) \otimes \B.$$
The main properties of UMD Banach spaces were established by Bourgain \cite{Bou}, Burkholder \cite{Bu1}
and Rubio de Francia \cite{Rub}.

Assume that $\{\gamma_j\}_{j=1}^\infty$ is a sequence of
independent standard normal variables defined on some probabilistic
space $(\Omega, \mathfrak{F},\mathbb{P})$. Let $H$ be a Hilbert space and
let $\B$ be a Banach space. We say that a linear operator
$T:H\rightarrow \B$ is $\gamma$-summing (shortly $T\in
\gamma^\infty(H,\B))$ when
$$\|T\|_{\gamma^\infty(H,\B)}
    = \sup \Big(\mathbb E \Big\|\sum_{j=1}^k \gamma_jT(h_j)\Big\|_\B^2 \Big)^{1/2} < \infty,$$
where the supremum is taken over all finite orthonormal family
$h_1,\ldots,h_k$ in $H$.
Here, by $\mathbb E$ we denote the
expectation with respect to $\mathbb{P}$. The space $\gamma^\infty(H,\B)$ becomes a
Banach space when it is endowed with the norm
$\| \cdot \|_{\gamma^\infty(H,\B)}$. The space of $\gamma$-radonifying
operators (shortly, $\gamma(H,\B)$) is the closure in
$\gamma^\infty(H,\B)$ of the subspace spanned by the finite rank
operators from $H$ into $\B$. According to \cite[Proposition 3.15]{Nee},
if $H$ is separable and $\{h_j\}_{j=1}^\infty$ is an
orthonormal basis in $H$, then $T\in \gamma(H,\B)$ if, and only if,
the series $\displaystyle\sum_{j=1}^\infty \gamma_jTh_j$ converges
in $L^2(\Omega,\B)$ and, in this case,
$$\|T\|_{\gamma^\infty(H,\B)}
    = \Big(\mathbb E \Big \|\sum_{j=1}^\infty \gamma_jTh_j \Big \|_\B^2\Big)^{1/2}.$$
We will write $\|T\|_{\gamma(H,\B)} =
\|T\|_{\gamma^\infty(H,\B)}$, $T\in \gamma(H,\B)$.

Hoffman-Jorgensen \cite{HJ} and Kwapie\'n \cite{Kw} established that
$\gamma(H,\B) =\gamma^\infty(H,\B)$ provided that the Banach space
$\B$ does not contain any closed subspace isomorphic to $c_0$. We
recall that UMD Banach spaces satisfy this property.

Suppose that $(M,\mathfrak{M},\mu)$ is a measure space and
$H=L^2(M,\mathfrak{M},\mu)$. A function $f:M\rightarrow \B$ is
said to be weakly $L^2$ when, for every $L\in \B^*$, the function
$L\circ f \in H$. Then, there exists a bounded and linear
operator $T_f:H\rightarrow \B$ (shortly $T_f \in
\mathcal{L}(H,\B)$) such that, for every $L\in \B^*$,
$$\langle L,T_f(h)\rangle_{\B^*,\B}
    =\int_M\langle L,f(t)\rangle_{\B^*,\B}h(t)d\mu(t),\;\,h\in H,$$
provided that $f$ is weakly $L^2$.
We say that $f \in \gamma(M,\mu;\B)$ provided that $T_f \in
\gamma(H,\B)$. If $\B$ does not contain any closed subspace
isomorphic to $c_0$, then $\gamma(M,\mu,\B)$ is a dense subspace
of $\gamma(H,\B)$ (\cite[Remark 2.16]{Kw}).

Assume that $\B$ is a UMD Banach space and $1<p<\infty$. Hyt\"onen,
van Neerven and Portal \cite[Definition 4.1]{HNP} defined the tent
space $T_p^2(\R^n,\B)$ as the completion of
$C_c^\infty(\R^n\times (0,\infty))\otimes \B$,
where $C_c^\infty(\R^n\times (0,\infty))$ denotes the space of smooth functions with compact support in $\R^n \times (0,\infty)$, with respect
to the norm
$$\|f\|_{T_p^2(\R^n,\B)}
    = \|Jf\|_{L^p(\R^n, \gamma(H,\B))},$$
where, from now on,
$$H=L^2\Big( \R^n \times (0,\infty), \frac{dydt}{t^{n+1}} \Big),$$
the functional $J$ is defined by
$$ J: f \rightarrow [x \rightarrow [(y,t)\rightarrow \chi_{B(x,t)}(y)f(y,t)]]$$
and $B(x,t)=\{y \in \R^n: |x-y| < t\}$, $ x \in \R^n$ and $t>0$.

By taking into account  that $\gamma(H,\mathbb C) \cong H$, it is clear that
$T_p^2(\R^n, \mathbb C)= T_p^2(\R^n)$. Then, the tent
space $T_p^2(\R^n,\B)$ can be seen as a Banach valued
extension of the classical tent space $T_p^2(\R^n)$. The
main properties of the space $T_p^2(\R^n,\B)$ were
established in \cite{HNP}, where Banach values tent spaces associated
with certain bisectorial operators were defined. An alternative and equivalent definition for tent spaces $T_p^2(\mathbb{R}^n,\mathbb{B})$ can be encountered in  \cite{Kemp}.

In \cite[Theorem 8.2]{HNP} (see also \cite[Example, Section 4]{Kemp})  it was proved a vectorial extension of
(\ref{1.1}) by using the tent spaces $T_p^2(\R^n,\B)$. In the
following we extend in some sense (by considering any
positive order of derivatives) the result in \cite[Theorem 8.2]{HNP}.

\begin{teo}\label{teo 1.1}
    Let $1<p<\infty$ and $\beta >0$. Assume that $\B$
    is a UMD Banach space. Then, there exists $C>0$ such that
    $$\frac{1}{C}\|f\|_{L^p(\R^n,\B)}
        \leq \|t^\beta\partial_t^\beta P_t(f)\|_{T_p^2(\R^n,\B)}
        \leq C\|f\|_{L^p(\R^n,\B)}, \quad f \in L^p(\R^n,\B).$$
\end{teo}
Since $T_p^2(\R^n,\mathbb C)= T_p^2(\R^n)$ Theorem
\ref{1.1} is an extension of Theorem~\ref{ThA} to Lebesgue-Bochner
spaces $L^p(\R^n,\B)$, $1<p<\infty$, provided that $\B$ is
a UMD. Note that a UMD Banach space is not necessarily isomorphic to
a Hilbert space.

We consider the Schr\"odinger operator $L_V$ in $\R^n$ defined by
$$L_V
    =-\Delta + V,$$
where $\Delta$ represents the usual Laplacian operator, that
is, $\Delta=\sum_{j=1}^n\frac{\partial^2}{\partial x_j^2}$. We
assume that the potential $V \not \equiv 0$ is a nonnegative measurable function
for which there exist $s > n/2$ and $C>0$ such that, for
every ball $B$ in $\R^n$,
\begin{equation}\label{1.2}
    \left(\int_BV(x)^sdx\right)^{1/s}
        \leq C\int_BV(x) dx.
\end{equation}
When $V$ satisfies (\ref{1.2}) we say that $V$ verifies the $s$-reverse
H\"older inequality and we write $V \in RH_s(\R^n)$.

In a precise way our Schr\"odinger operator $\mathfrak{L}_V$ is defined as follows. We consider the sesquilinear form $Q_V$ given by
$$Q_V[f,g]
    =\int_{\R^n}\nabla f(x)\overline{\nabla g(x)}dx + \int_{\R^n}V(x) f(x)\overline{g(x)}dx, \quad\; (f,g) \in D(Q_V),$$
where $\nabla$ denotes the usual gradient. The domain $D(Q_V)$ of $Q$ is the product $\mathfrak{D}_V \times \mathfrak{D}_V$, where
$$\mathfrak{D}_V
    =\{f \in L^2(\R^n): |\nabla f| \in L^2(\R^n)\;\; \mbox{and}\;\;V^{1/2}f \in L^2(\R^n)\}.$$
The Schr\"odinger operator $\mathfrak{L}_V$ is the unique selfadjoint operator such that its domain is $\mathfrak{D}_V$
and
$$\langle \mathfrak{L}_Vf,g \rangle
    = Q_V[f,g], \quad f,g \in \mathfrak{D}_V.$$
It is clear that $C_c^\infty(\R^n)$, the space of smooth functions with compact support in $\R^n$,
is contained in $\mathfrak{D}_V$ and $\mathfrak{L}_V=L_V$ on $C_c^\infty(\R^n)$.
$\mathfrak{L}_V$ is a positive operator.

The semigroup of operators $\{W_t^{\mathfrak{L}_V}\}_{t>0}$ generated by $-\mathfrak{L}_V$ in $L^2(\R^n)$ can be written as
$$W_t^{\mathfrak{L}_V}(f)
    = \int_{[0,\infty)}e^{-t\lambda} E_{\mathfrak{L}_V}(d\lambda)f, \quad f \in L^2(\R^n)\;\;\mbox{and}\;\; t>0,$$
where $E_{\mathfrak{L}_V}$ denotes the spectral measure for $\mathfrak{L}_V$.

For every $t>0$ there exists a measurable function $W_t^{\mathfrak{L}_V}(x,y)$, $x,y \in \R^n$, such that for each $f\in L^2(\R^n)$,
\begin{equation}\label{1.3}
    W_t^{\mathfrak{L}_V}(f)(x)=\int_{\R^n}W_t^{\mathfrak{L}_V}(x,y)f(y)dy.
\end{equation}
Moreover, according to the Feynman-Kac formula (\cite[p. 280]{DZ1}), we have that
$$\left|W_t^{\mathfrak{L}_V}(x,y)\right|
    \leq C\frac{e^{-|x-y|^2/4t}}{t^{n/2}},\,\,x,y \in \R^n\;\;\mbox{and}\;\; t>0.$$
Then, the integral in (\ref{1.3}) is absolutely convergent for every $f\in L^p(\R^n)$, $1 \leq p\leq\infty$.
The family $\{W_t^{\mathfrak{L}_V}\}_{t>0}$, where $W_t^{\mathfrak{L}_V}$, $t>0$, is defined by (\ref{1.3}), is
a positive bounded semigroup in $L^p(\R^n)$, $1 \leq p < \infty$, generated by $-\mathfrak{L}_V$.
$\{W_t^{\mathfrak{L}_V}\}_{t>0}$ is not Markovian because $V \not \equiv 0$.

The semigroup of operators $\{P_t^{\mathfrak{L}_V}\}_{t>0}$ subordinated to $\{W_t^{\mathfrak{L}_V}\}_{t>0}$,
also called Poisson semigroup associated to $\mathfrak{L}_V$, is defined by
$$P_t^{\mathfrak{L}_V}(f)(x)
    =\frac{t}{2\sqrt{\pi}}\int_0^\infty \frac{e^{-t^2/4u}}{u^{3/2}}W_u^{\mathfrak{L}_V}(f)(x)du, \quad f \in L^p(\R^n), \ 1 \leq p < \infty\text{ and } t>0.$$
An  important special case of Schr\"odinger operator is the Hermite operator $\mathcal{H}$
(also called harmonic oscillator) that appears when $V(x)=|x|^2$, $x\in \R^n$.

Harmonic analysis associated with Schr\"odinger and Hermite
operators has been developed in the last years by several authors
(\cite{AB},
\cite{BCFST},
\cite{BHS2},
\cite{BT},
\cite{DGMTZ},
\cite{DG},
\cite{DZ1},
\cite{DZ3},
\cite{Sh},
\cite{ST1} and
\cite{Th},  amongst others).

Our second result establishes the equivalence in Theorem \ref{1.1}
when the classical Poisson semigroup is replaced by the Poisson semigroups
$\{P_t^{\mathfrak{L}_V}\}_{t>0}$ or $\{P_t^{\mathcal{H}}\}_{t>0}$.

\begin{teo}\label{Theorem 1.2}
    Let $1<p<\infty$ and $\beta>0$. Assume that $\B$ is a UMD Banach space.
    \begin{itemize}
    \item[$(i)$] If $V \in RH_s(\R^n)$ for some $s>n/2$, and $n \geq 3$, then there exists $C>0$ such that
    $$\frac{1}{C}\|f\|_{L^p(\R^n,\B)}
        \leq \|t^\beta \partial_t^\beta P_t^{\mathfrak{L}_V}(f)\|_{T_p^2(\R^n,\B)}
        \leq C\|f\|_{L^p(\R^n,\B)}, \;\;f\in L^p(\R^n,\B).$$
    \item[$(ii)$] For every $n \in \mathbb N$, there exists $C>0$ such that
    $$\frac{1}{C}\|f\|_{L^p(\R^n,\B)}
        \leq \|t^\beta\partial_t^\beta P_t^{\mathcal{H}}(f)\|_{T_p^2(\R^n,\B)}
        \leq C\|f\|_{L^p(\R^n,\B)}, \;\;f\in L^p(\R^n,\B).$$
    \end{itemize}
\end{teo}
Since $T_p^2(\R^n,\mathbb C) = T_p^2(\R^n)$, $1 <p<\infty$,
the scalar results can be deduced as special cases of Theorem \ref{Theorem 1.2}.

We now define the space $T_p^2((0,\infty),\B)$, $1<p<\infty$,
where $\B$ is again a UMD Banach space. Let
$1<p<\infty$. The tent space $T_p^2((0,\infty),\B)$ is the
completion of $C_c^\infty(0,\infty) \otimes \B$ with respect to the norm
$$\|f\|_{T^2_p((0,\infty),\B)}
    = \|J_+f\|_{L^p((0,\infty),\gamma(H_+,\B))},$$
where, from now on,
$$H_+=L^2\Big( (0,\infty)^2, \frac{dydt}{t^2} \Big),$$
the functional $J_+$ is defined by
$$J_+:f \rightarrow [x\rightarrow[(y,t)\rightarrow \chi_{B_+(x,t)}(y) f(y,t)]]$$
and $B_+(x,t)=\{y \in (0,\infty):|x-y|<t\}$, $x,t\in(0,\infty)$.
Here, by $C_c^\infty(0,\infty)$ we denote the space of smooth functions with compact support on $(0,\infty)$.

We will use the tent space $T_p^2((0,\infty),\B)$ to get
equivalent norms in the Lebesgue-Bochner space
$L^p((0,\infty),\B)$, for every $1<p<\infty$ and every UMD
Banach space $\B$. Our new norms (see Theorems \ref{teo 1.3} and
\ref{teo 1.4} below) involve Poisson semigroups associated with
Bessel and Laguerre operators.

Harmonic analysis in the Bessel setting began with the deep paper
of Muckhenhoupt and Stein \cite{MS}. Recently, operators related
to the harmonic analysis (Riesz transform, Littlewood-Paley
functions, maximal operators, multipliers,\dots) in the Bessel
context have been investigated (see, for instance, \cite{BCR2},
\cite{BFMR}, \cite{BFMT}, \cite{BHNV} and \cite{Vill}). We
consider the Bessel operator on $(0,\infty)$,
$$B_\lambda
    = -x^\lambda Dx^{2\lambda}Dx^{-\lambda}
    = -\frac{d^2}{dx^2} + \frac{\lambda(\lambda-1)}{x^2},$$
where $\lambda>0$. The Hankel transform $h_\lambda$ is defined by
$$h_\lambda(f)(x)
    = \int_0^\infty\sqrt{xy}J_{\lambda-1/2}(xy) f(y) dy, \quad\;f \in L^1(0,\infty),$$
where  $J_\alpha$ denotes the Bessel function of the  first kind
and order $\alpha$. $h_\lambda$ plays with respect to the Bessel
operator the same role as the Fourier transform with respect to
the classical Laplacian operator. $h_\lambda$ can be extended from
$L^1(0,\infty) \cap L^2(0,\infty)$ to $L^2(0,\infty)$ as an
isometry in $L^2(0,\infty)$ (\cite[Ch. VIII]{Ti2}). By using
well-known properties of the Bessel function  $J_\alpha$ we can
deduce that, for every $f\in C_c^\infty(0,\infty)$,
$$h_\lambda(B_\lambda f)(x)
    =x^2h_\lambda(f)(x), \quad x \in (0,\infty),$$
(\cite[Lemma 5.4-1(5)]{Ze2}). We define the operator $\mathfrak{B}_\lambda$ by
$$\mathfrak{B}_\lambda(f)
    = h_\lambda(x^2h_\lambda(f)), \quad f \in D(\mathfrak{B}_\lambda),$$
where the domain $D(\mathfrak{B}_\lambda)$ of $\mathfrak{B}_\lambda$ is given by
$$D(\mathfrak{B}_\lambda)
    =\{f\in L^2(0,\infty): y^2h_\lambda(f) \in L^2(0,\infty)\}.$$
Since $h_\lambda^{-1} = h_\lambda$ in $L^2(0,\infty)$,
$C_c^\infty(0,\infty) \subset D(\mathfrak{B}_\lambda)$ and
$\mathfrak{B}_\lambda f= B_\lambda f$, $f \in
C_c^\infty(0,\infty)$. The operator $-\mathfrak{B}_\lambda$
generates a positive and bounded semigroup of operators
$\{W_t^{\mathfrak{B}_\lambda}\}_{t>0}$ in $L^p(0,\infty)$, for
every $1\leq p < \infty$. Moreover, the Poisson semigroup
$\{P_t^{\mathfrak{B}_\lambda}\}_{t>0}$ associated with the Bessel
operator $\mathfrak{B}_\lambda$ can be written as
$$P_t^{\mathfrak{B}_\lambda}(f)(x)
    = \int_0^\infty P_t^{\mathfrak{B}_\lambda}(x,y) f(y) dy, \quad t\in (0,\infty),$$
for every $f \in L^p(0,\infty)$, $1 \leq p<\infty$. Here the Poisson kernel
$P_t^{\mathfrak{B}_\lambda}(x,y)$, $t,x,y \in (0,\infty)$, is given by (\cite[ (16.4)]{MS}, \cite{Wei})
$$P_t^{\mathfrak{B}_\lambda}(x,y)
    = \frac{2\lambda(xy)^\lambda t}{\pi}\int_0^\pi \frac{(\sin\theta)^{2\lambda-1}}{(t^2+ (x-y)^2+2xy(1-\cos\theta))^{\lambda+1}}d\theta, \quad t,x,y\in (0,\infty).$$

\begin{teo}\label{teo 1.3}
    Let $1<p<\infty$ and $\beta,\lambda >0$. Assume that $\B$ is a UMD Banach space. Then, there exists $C>0$ such that
    $$\frac{1}{C}\|f\|_{L^p((0,\infty),\B)}
        \leq \|t^\beta \partial_t^\beta P_t^{\mathfrak{B}_\lambda}(f)\|_{T_p^2((0,\infty),\B)}
        \leq C\|f\|_{L^p((0,\infty),\B)}, \quad f \in L^p((0,\infty),\B).$$
\end{teo}

Muckenhoupt (\cite{Mu2} and \cite{Mu3}) began the study of
harmonic analysis associated to Laguerre operators. Later, Dinger
\cite{Di} established $L^p$-boundedness properties for the maximal
operator defined by the $n$-dimensional heat semigroup in the
Laguerre context. In the last years a lot of authors have
investigated harmonic analysis operators related to Laguerre
operators (see, for instance,
\cite{BFRST1},
\cite{Dzi2},
\cite{FSS},
\cite{GIT},
\cite{HTV},
\cite{NoSt2},
\cite{Nee}, and
\cite{Sz}).

We consider the Laguerre operator on $(0,\infty)$
$$L_\alpha
    =-\frac{d^2}{dx^2}+ \frac{\alpha^2-1/4}{x^2}+x^2,$$
where $\alpha >-1/2$. For every $k\in \mathbb N$, we have that
$$L_\alpha\varphi_k^\alpha
    = 2 (2k+\alpha+1)\varphi_k^\alpha,$$
where
$$\varphi_k^\alpha(x)
    =\left( \frac{2\Gamma(k+1)}{\Gamma(k+\alpha+1)}\right)^{1/2} e^{-x^2/2} x^{\alpha+1/2}\ell_k^\alpha(x^2), \quad x \in (0, \infty),$$
and $\ell_k^\alpha$ represents the $k$-th Laguerre polynomial of order $\alpha$ (\cite[p. 100--102]{Sze}).
The sequence $\{\varphi_k^\alpha\}_{k\in\mathbb N}$ is an orthonormal basis in $L^2(0,\infty)$.

We define the operator $\mathcal{L}_\alpha$ as follows
$$\mathcal{L}_\alpha f
    = 2 \sum_{k=0}^\infty (2k+\alpha+1) c_k^\alpha(f)\varphi_k^\alpha, \quad f\in D(\mathcal{L}_\alpha),$$
where, for every $k\in \mathbb N$,
$$c_k^\alpha(f)
    =\int_0^\infty \varphi_k^\alpha(x) f(x) dx, \quad f \in L^2(0,\infty),$$
and the domain $D(\mathcal{L}_\alpha)$ of $\mathcal{L}_\alpha$ is given by
$$D(\mathcal{L}_\alpha)
    =\{f \in L^2(0,\infty): \sum_{k=0}^\infty (2k+\alpha+1)^2|c_k^\alpha(f)|^2 < \infty\}.$$
The operator $-\mathcal{L}_\alpha$ generates a positive and bounded semigroup $\{W_t^{\mathcal{L}_\alpha}\}_{t>0}$
in $L^2(0, \infty)$, given by
$$W_t^{\mathcal{L}_\alpha}(f)
    = \sum_{k=0}^\infty e^{-2(2k+\alpha+1)}c_k^\alpha(f)\varphi_k^\alpha, \quad f \in L^2(0,\infty) \;\;\mbox{and}\;\; t>0.$$
Mehler's formula for Laguerre polynomials (\cite[p. 8]{Th}) allows us to write, for each $f \in L^2(0,\infty)$,
\begin{equation}\label{1.4}
    W_t^{\mathcal{L}_\alpha}(f)(x)
        = \int_0^\infty W_t^{\mathcal{L}_\alpha}(x,y)f(y)dy, \quad t,x \in (0,\infty),
\end{equation}
where
\begin{equation}\label{eq:heatLag}
    W_t^{\mathcal{L}_\alpha}(x,y)
        = \left(\frac{2e^{-2t}}{1-e^{-4t}}\right)^{1/2}\left(\frac{2xye^{-2t}}{1-e^{-4t}}\right)^{1/2}
        I_\alpha\left(\frac{2xye^{-2t}}{1-e^{-4t}}\right)\exp\left(-\frac{1}{2}(x^2+y^2)\frac{1+e^{-4t}}{1-e^{-4t}}\right),
\end{equation}
and $I_\alpha$ represents the modified Bessel function of the first kind and order $\alpha$.

Moreover, if $W_t^{\mathcal{L}_\alpha}$ is defined by (\ref{1.4}), for every $t>0$,
then $\{W_t^{\mathcal{L}_\alpha}\}_{t>0}$ is a positive and bounded semigroup of operators in $L^p(0,\infty)$, $ 1\leq p < \infty$.

As usual the Poisson semigroup $\{P_t^{\mathcal{L}_\alpha}\}_{t>0}$ associated with $\mathcal{L}_\alpha$
is defined as the one subordinated to $\{W_t^{\mathcal{L}_\alpha}\}_{t>0}$, that is, for every $t>0$,
\begin{equation}\label{eq:PoissLag}
    P_t^{\mathcal{L}_\alpha}(f)(x)
        = \frac{t}{2\sqrt{\pi}}\int_0^\infty \frac{e^{-t^2/4u}}{u^{3/2}}W_u^{\mathcal{L}_\alpha}(f)(x) du, \quad f \in L^p(0,\infty)\;\;\mbox{and}\;\; 1\leq p < \infty.
\end{equation}

\begin{teo} \label{teo 1.4}
    Let $1<p<\infty$ and $\alpha, \beta >0$.
    Assume that $\B$ is a UMD Banach space. Then, there exists $C>0$ such that
    $$\frac{1}{C} \|f\|_{L^p((0,\infty),\B)}
        \leq \|t^\beta\partial_t^\beta P_t^{\mathcal{L}_\alpha}(f)\|_{T_p^2((0,\infty),\B)}
        \leq C\|f\|_{L^p((0,\infty),\B)}, \quad f \in L^p((0,\infty), \B).$$
\end{teo}
In the following sections we prove Theorems \ref{teo 1.1}, \ref{Theorem 1.2}, \ref{teo 1.3} and \ref{teo 1.4}.

Throughout this paper by $C$ and $c$ we always denote positive constants that can change in each occurrence.

\section{Proof of Theorem \ref{teo 1.1}}\label{sec:Laplacian}

We split the proof of Theorem~\ref{teo 1.1} in the Lemmas~\ref{Lem2.1} and \ref{Lem2.2} below.
We are going to use the arguments developed in \cite[Theorem 8.2]{HNP}.

\begin{lema}\label{Lem2.1}
    Let $\B$ be a UMD Banach space, $1<p<\infty$ and $\beta>0$.
    Then, there exists $C>0$ such that
    \begin{equation}\label{2.0.1}
        \|t^\beta \partial _t^\beta P_t(f)\|_{T_p^2(\mathbb{R}^n,\B)}
            \leq C\|f\|_{L^p(\mathbb{R}^n,\B)},\quad f \in L^p(\mathbb{R}^n,\B).
    \end{equation}
\end{lema}

\begin{proof}
    Let $f\in L^p(\mathbb{R}^n,\B)$. It is not hard to see that, for every $k\in \mathbb{N}$,
    $$\partial _t^k P_t(f)(x)
        = \int_{\mathbb{R}^n} \partial _t^k P_t(x-y) f(y)dy,\quad x \in \mathbb{R}^n \mbox{ and } t>0.$$
    Assume that $m\in \mathbb{N}$ is such that $m-1\leq \beta <m$. We have that
    \begin{align*}
        \partial _t^\beta P_t(f)(x)
            &= \frac{e^{-i\pi (m-\beta )}}{\Gamma (m-\beta )}\int_0^\infty s^{m-\beta -1}\int_{\mathbb{R}^n}\partial _t^m P_{t+s}(x-y) f(y)dyds\\
            &= \frac{e^{-i\pi (m-\beta )}}{\Gamma (m-\beta )}\int_{\mathbb{R}^n}f(y)\int_0^\infty s^{m-\beta -1}\partial _t^m P_{t+s}(x-y) dsdy\\
            &= \int_{\mathbb{R}^n} \partial _t^\beta P_t(x-y) f(y) dy,\quad x \in \mathbb{R}^n \mbox{ and } t>0.
    \end{align*}
    The interchange in the order of integration is justified because
    \begin{equation}\label{2.0}
        \int_0^\infty \int_{\mathbb{R}^n}\|f(y)\|_\B s^{m-\beta -1} |\partial _t^m P_{t+s}(x-y)|dyds<\infty,\quad x \in \mathbb{R}^n \mbox{ and } t>0.
    \end{equation}
    Indeed, according to Fa di Bruno's formula (\cite[(4.6)]{GLLNU}) we can write,
    for each $n\geq 2$, $t>0$ and $z\in \mathbb{R}^n$,
    \begin{align}\label{2.1}
        &\partial _t^m\left[\frac{t}{(t^2+|z|^2)^{(n+1)/2}}\right]
            = \frac{1}{1-n}\partial _t^{m+1}\left[\frac{1}{(t^2+|z|^2)^{(n-1)/2}}\right]\nonumber\\
        & \quad = \frac{1}{1-n}\sum_{\ell=0}^{\lfloor (m+1)/2 \rfloor}
            \Big(\frac{-1}{2}\Big)^{m+1-\ell} (n-1)(n+1) \cdot \ldots \cdot (n-1+2(m-\ell))E_{m,\ell}
            \frac{t^{m+1-2\ell}}{(t^2+|z|^2)^{(n+1+2(m-\ell))/2}},
    \end{align}
    and
    \begin{align}\label{2.2}
        \partial _t^m\left[\frac{t}{t^2+|z|^2}\right]
            &= \frac{1}{2}\partial _t^{m+1}\ln (t^2+|z|^2)
            = \frac{1}{2} \sum_{\ell=0}^{\lfloor (m+1)/2 \rfloor} (-1)^{m-\ell}(m-\ell)! E_{m,\ell}
                 \frac{t^{m+1-2\ell}}{(t^2+|z|^2)^{m+1-\ell }},
    \end{align}
    where $E_{m,\ell }=2^{m+1-2\ell }(m+1)!/(\ell !(m+1-2\ell )!)$.
    From (\ref{2.1}) and \eqref{2.2} we deduce that
    $$\left|\partial _t^m\left[\frac{t}{(t^2+|z|^2)^{(n+1)/2}}\right]\right|
        \leq C\frac{1}{(t+|z|)^{m+n}},\quad z \in \mathbb{R}^n \mbox{ and } t>0.$$
    Then,
    \begin{align*}
        \int_0^\infty s^{m-\beta -1}\left|\partial _t^m\frac{t+s}{((t+s)^2+|x-y|^2)^{(n+1)/2}}\right|ds
            &\leq C\int_0^\infty \frac{s^{m-\beta -1}}{(t+s+|x-y|)^{m+n}}ds\\
            &\leq C\frac{1}{(t+|x-y|)^{n+\beta }},\quad x,y\in \mathbb{R}^n, \ t>0,
    \end{align*}
    and H\"older inequality allows us to obtain (\ref{2.0}).

    Also, we have that
    \begin{equation}\label{2.3}
        |t^\beta \partial _t^\beta P_t(x-y)|
            \leq C\frac{t^\beta }{(t+|x-y|)^{n+\beta }},\quad x,y\in \mathbb{R}^n, \ t>0.
    \end{equation}

    To simplify we write
    \begin{equation*}\label{2.3.1}
        (Sf)(t,y)
            =\int_{\mathbb {R}^n}k(y,z,t)f(z)dz,\quad f\in L^p(\mathbb{R}^n,\B),
    \end{equation*}
    where $k(y,z,t)=t^\beta \partial_t^\beta P_t(y-z)$, $y,z\in \mathbb{R}^n$ and $t>0$.
    Our objective is to see that $S$ is a
    bounded operator from $L^p(\mathbb{R}^n,\B)$ into
    $T_p^2(\mathbb{R}^n,\B)$. In order to do this we use \cite[Theorem 4.8]{HNP}.

    We consider the operator
    $$(\mathbb{S}g)(t,y)
        =\int_{\mathbb{R}^n}k(y,z,t)g(z)dz,\quad g\in L^2(\mathbb{R}^n).$$

    As usual, for every $g\in L^2(\mathbb{R}^n)$, we denote by $\widehat{g}$ the Fourier transform of
    $g$ and by $\widecheck{g}$ the inverse Fourier transform of $g$. Let $g\in L^2(\mathbb{R}^n)$. It is well-known that
    \begin{equation*}\label{2.0.2}
        \partial _t^kP_t(g)(y)
            =((-1)^k|z|^k e^{-t|z|}\widehat{g})^{\widecheck{\;}},\quad k\in \mathbb{N}.
    \end{equation*}
    Since $(-1)^k|z|^k e^{-t|z|}\widehat{g}\in L^1(\mathbb{R}^n)$, $t>0$ and $k\in \mathbb{N}$, we can write
    $$\partial _t^kP_t(g)(y)
        =\frac{(-1)^k}{(2\pi )^{n/2}}\int_{\mathbb{R}^n}e^{iyz}|z|^ke^{-t|z|}\widehat{g}(z)dz.$$
    Hence,
    \begin{align}\label{eqB}
        \partial_t^\beta P_t(g)(y)
            &=\frac{(-1)^me^{-i\pi (m-\beta )}}{(2\pi )^{n/2}\Gamma (m-\beta )}\int_0^\infty s^{m-\beta -1}\int_{\mathbb{R}^n}e^{iyz} |z|^m e^{-(t+s)|z|}\widehat{g}(z)dzds \nonumber\\
            &=\frac{e^{i\pi \beta }}{(2\pi )^{n/2}\Gamma (m-\beta )}\int_{\mathbb{R}^n}e^{iyz}|z|^m e^{-t|z|} \widehat{g}(z) \int_0^\infty e^{-s|z|}s^{m-\beta -1}dsdz \nonumber\\
            &=\frac{e^{i\pi \beta}}{(2\pi )^{n/2}}\int_{\mathbb{R}^n}e^{iyz} |z|^\beta e^{-t|z|} \widehat{g}(z)dz
            =e^{i\pi \beta}(|z|^\beta e^{-t|z|} \widehat{g}(z))^{\widecheck{\;}}.
    \end{align}
    The interchange of the order of integration is justified by the absolute convergence of the integral.
    Thus, Plancherel equality leads to
    \begin{align*}\label{eqC}
        \|\mathbb{S} g\|^2_{T_2^2(\mathbb{R}^n)}
            &=\int_{\mathbb{R}^n}\int_{\Gamma (x)}|(\mathbb{S}g)(y,t)|^2\frac{dydt}{t^{n+1}}dx
             =\int_{\mathbb{R}^n}\int_{\Gamma (x)}|t^\beta \partial _t^\beta P_t(g)(y)|^2  \frac{dydt}{t^{n+1}}dx \nonumber \\
            &=\int_{\mathbb{R}^n}\int_{\Gamma (x)}|(e^{-t|z|}(t|z|)^\beta \widehat{g}(z))^{\widecheck{\;}}(y)|^2\frac{dydt}{t^{n+1}}dx \nonumber \\
            &=\int_0^\infty \int_{\mathbb{R}^n}\int_{B(y,t)}|(e^{-t|z|}(t|z|)^\beta \widehat{g}(z))^{\widecheck{\;}}(y)|^2\frac{dxdydt}{t^{n+1}} \nonumber \\
            &=v_n\int_0^\infty \int_{\mathbb{R}^n}|(e^{-t|z|}(t|z|)^\beta \widehat{g}(z))^{\widecheck{\;}}(y)|^2\frac{dydt}{t}
            =v_n\int_0^\infty \int_{\mathbb{R}^n}|e^{-t|z|}(t|z|)^\beta \widehat{g}(z)|^2\frac{dzdt}{t} \nonumber \\
            &=v_n \frac{\Gamma (2\beta )}{2^{2\beta }}\int_{\mathbb{R}^n}|\widehat{g}(z)|^2dz
            =v_n\frac{\Gamma (2\beta )}{2^{2\beta }}\|g\|_{L^2(\mathbb{R}^n)}^2,
    \end{align*}
    where $v_n$ is the volume of the unit ball in $\mathbb{R}^n$.
    Hence, $\mathbb{S}$ is a bounded operator from $L^2(\mathbb{R}^n)$ into $T_2^2(\mathbb{R}^n)$.

    We now prove that
    \begin{equation}\label{2.4}
        |\nabla_zk(y,z,t)|
            \leq C\frac{t^\beta}{(t+|y-z|)^{n+\beta +1}},\quad y,z\in \mathbb{R}^n \mbox{ and } t>0.
    \end{equation}
    Let $j=1,\dots,n$. According to (\ref{2.1}) and (\ref{2.2}) we obtain
    $$\partial_j\partial _t^m\left[\frac{t}{(t^2+|z|^2)^{(n+1)/2}}\right]
        =\sum_{\ell=0}^{\lfloor (m+1)/2 \rfloor} a_l\frac{t^{m+1-2l}z_j}{(t^2+ |z|^2)^{(n+3+2(m-l))/2}},\quad z \in \mathbb{R}^n \mbox{ and } t>0,$$
    for certain $a_l\in \mathbb{R}$. Then,
    $$\left|t^\beta \partial_j\partial _t^m\left[\frac{t}{(t^2+|z|^2)^{(n+1)/2}}\right]\right|
        \leq C \frac{t^\beta }{(t+|z|)^{n+m+1}},\quad z \in \mathbb{R}^n \mbox{ and } t>0.$$
    By proceeding as in the proof of (\ref{2.3}) we get
    $$|\partial _jk(y,z,t)|
        \leq C\frac{t^\beta }{(t+|y-z|)^{n+1+\beta }},\quad y,z \in \mathbb{R}^n \mbox{ and } t>0.$$
    Thus (\ref{2.4}) is established.

    From (\ref{2.4}) and by using the mean value theorem we can deduce that
    \begin{equation*}\label{2.5}
        |k(y,z,t)-k(y,z',t)|\leq C\frac{t^\beta |z-z'|}{(t+|y-z|)^{n+\beta+1}},
    \end{equation*}
    provided that $t>0$, $y,z,z'\in \mathbb{R}^n$ and $|y-z|+t>2|z-z'|$.

    Moreover, we can write
    \begin{equation*}\label{2.6}
        \int_{\mathbb{R}^n}k(y,z,t)dz=t^\beta \partial _t^\beta \int_{\mathbb{R}^n}P_t(y-z)dz=0,\quad t>0 \mbox{ and }y\in \mathbb{R}^n.
    \end{equation*}

    According to  \cite[Theorem 4.8]{HNP}, the operator
    $S_\B=\mathbb{S}\otimes I_\B$ can be extended from $C_c^\infty (\mathbb{R}^n)\otimes \B$
    to $L^p(\mathbb{R}^n,\B)$ as a bounded
    operator $\widetilde{S_\B}$ from $L^p(\mathbb{R}^n,\B)$ into
    $T_p^2(\mathbb{R}^n,\B)$.
    Also, as a special case, the operator $\mathbb{S}$ can be extended from
    $C_c^\infty(\R^n)$ to $L^p(\R^n)$ as a bounded operator $\widetilde{S_\C}$
    from $L^p(\R^n)$ into $T_p^2(\R^n)$.
    We are going to show that $S=\widetilde{S_\B}$ on $L^p(\mathbb{R}^n,\B)$.

    In order to get a better understanding of the proof we see firstly that
    $\widetilde{S_\C}=\mathfrak{S}$ on $L^p(\R^n)$, where
    $$(\mathfrak{S}g)(t,y)
        = \int_{\R^n} k(y,z,t) g(z)dz, \quad g \in L^p(\R^n).$$
    Let $f \in L^p(\mathbb{R}^n)$. We choose a sequence $(f_k)_{k\in \mathbb{N}}$
    in $C_c^\infty (\mathbb{R}^n)$ such that
    $$ f_k\longrightarrow f, \quad \text{as } k \to \infty, \text{ in } L^p(\mathbb{R}^n).$$
    Then,
    $$ \mathbb{S}f_k \longrightarrow \widetilde{S_\C}f, \quad \text{as } k\rightarrow \infty, \text{ in } T_p^2(\mathbb{R}^n).$$
    Hence, $(J(\mathbb{S}f_k))_{k\in \mathbb{N}}$ converges to a function $g\in L^p(\mathbb{R}^n,H)$.
    There exists an increasing sequence $(k_\ell)_{\ell \in \mathbb{N}}\subset \mathbb{N}$ and a subset
    $\Omega \in \mathbb{R}^n$ such that
    $|\mathbb{R}^n\setminus \Omega|=0$ and, for every $x\in \Omega$,
    \begin{equation}\label{2.6.1}
        [J(\mathbb{S}f_{k_\ell})](x)\longrightarrow g(x),\mbox{ as } \ell \rightarrow \infty , \mbox{ in } H.
    \end{equation}
    On the other hand, according to (\ref{2.3}) we have that, for every $\ell \in \N$,
    \begin{align}\label{2.7}
        \Big| S(f-f_{k_\ell})(t,x) \Big|
            & \leq\int_{\mathbb{R}^n}|f(y)-f_{k_\ell}(y)||t^\beta \partial _t^\beta P_t(y-x)|dy \nonumber \\
            &\leq C\int_{\mathbb{R}^n}|f(y)-f_{k_\ell}(y)|\frac{t^\beta }{(t+|x-y|)^{n+\beta }}dy\nonumber\\
            &\leq C\|f-f_{k_\ell}\|_{L^p(\mathbb{R}^n)}t^\beta \left(\int_{\mathbb{R}^n}\frac{1}{(t+|z|)^{(n+\beta )p'}}dz\right)^{1/p'}\nonumber\\
            &\leq C t^{-n/p}\|f-f_{k_\ell}\|_{L^p(\mathbb{R}^n)},\quad x\in \mathbb{R}^n \mbox{ and } t>0.
    \end{align}
    Here $p'=p/(p-1)$. From (\ref{2.7}) we deduce that, for each $\ell \in \N$ and $\eps>0$,
    \begin{align*}
        \Big\|[JS(f_{k_\ell }-f)](x)\Big\|_{L^2(\mathbb{R}^n\times (\varepsilon ,\infty ),\frac{dydt}{t^{n+1}})}^2
            &\leq C\|f_{k_\ell}-f\|_{L^p(\mathbb{R}^n)}^2\int_\varepsilon ^\infty \int_{B(x,t)}t^{-2n/p}\frac{dydt}{t^{n+1}}\\
            &\leq C\|f_{k_\ell}-f\|_{L^p(\mathbb{R}^n)}^2\varepsilon ^{-2n/p}, \quad x\in \mathbb{R}^n.
    \end{align*}
    Then, for every $\varepsilon >0$ and $x\in \mathbb{R}^n$,
    \begin{equation}\label{2.8}
        J(\mathbb{S}f_{k_\ell})(x)\longrightarrow J(\mathfrak{S}f)(x), \quad \mbox{as }\ell \rightarrow \infty, \mbox{ in }L^2\Big(\mathbb{R}^n\times (\varepsilon ,\infty ),\frac{dydt}{t^{n+1}}\Big).
    \end{equation}
    From (\ref{2.6.1}) and (\ref{2.8}) it follows that, for every $x\in \Omega$, $J(\mathfrak{S}f)(x)=g(x)$
    as elements of $H$.
    We conclude that
    $$J(\mathbb{S}f_k)\longrightarrow J(\mathfrak{S}f), \quad  \text{as } k\rightarrow \infty,
        \text{ in } L^p(\mathbb{R}^n,H).$$
    Thus, we prove that
    $$J(\mathbb{S}f_k)\longrightarrow J(\mathfrak{S}f), \quad  \text{as } k\rightarrow \infty,
        \text{ in } L^p(\mathbb{R}^n,\gamma(H,\B)).$$
    Hence, $\mathfrak{S}f=\widetilde{S_\C}f$.

    Almost the same proof works for every $f\in L^p(\mathbb{R}^n,\B)$. Let $f\in L^p(\mathbb{R}^n,\B)$.
    We choose a sequence $(f_k)_{k\in \mathbb{N}}\subset C_c^\infty (\mathbb{R}^n)\otimes \B$ such that
    $$f_k\longrightarrow f, \quad \text{as } k\rightarrow \infty, \text{ in } L^p(\mathbb{R}^n,\B).$$
    Then,
    $$ S_\B f_k \longrightarrow \widetilde {S_\B}f, \quad \text{as } k\rightarrow \infty, \text{ in } T_p^2(\mathbb{R}^n,\B).$$
    Hence, there exists $G\in L^p(\mathbb{R}^n,\gamma(H, \B))$ such that
    $$J(S_\B f_k)\longrightarrow G,\quad \mbox{ as }k\rightarrow \infty,\mbox{ in }L^p(\mathbb{R}^n, \gamma(H,\B)).$$
    There exists an increasing sequence $(k_\ell )_{\ell\in \mathbb{N}}\subset \mathbb{N}$
    and a set $\Omega \subset \mathbb{R}^n$ such that $|\mathbb{R}^n\setminus \Omega|=0$ and
    $$ J(S_\B f_{k_\ell })(x)\longrightarrow G(x), \quad \text{as } \ell \rightarrow \infty, \text{ in } \gamma(H,\B),$$
    for every $x\in \Omega$.

    By proceeding as above we can see that, for every $\varepsilon >0$ and $x\in \mathbb{R}^n$,
    $$J(S_\B f_{k_\ell })(x)\longrightarrow J(Sf)(x),\mbox{ as }\ell \rightarrow \infty, \text{ in } L^2\Big(\mathbb{R}^n\times (\varepsilon ,\infty ),\frac{dydt}{t^{n+1}};\B\Big).$$

    Since $\gamma(H,\B)$  is continuously contained in the space $\mathcal{L}(H,\B)$
    of bounded linear operators from $H$ into $\B$, we have that, for every $x\in \Omega$,
    $$J(S_\B f_{k_\ell })(x)\longrightarrow G(x), \quad \mbox{ as }\ell \rightarrow \infty , \mbox{ in }\mathcal{L}(H,\B).$$
    Assume that $h\in C_c^\infty (\mathbb{R}^n\times (0,\infty ))$ and $L\in \B^*$.
    Then, for every $x\in \Omega$,
    $$\langle L,[J(S_\B f_{k_\ell})(x)](h)\rangle _{\B^*,\B}\longrightarrow \langle L, [G(x)](h)\rangle_{\B^*,\B},\quad \mbox{ as }\ell \rightarrow \infty,$$
being
$$ 
[J(S_\B f_{k_\ell})(x)](h)=\int_{\mathbb{R}^n\times (0,\infty )}[J(S_\B f_{k_\ell})(x)](y,t)h(y,t)\frac{dydt}{t^{n+1}},\quad \ell \in \mathbb{N}.
$$

    Suppose that $g=\sum_{m=1}^s b_mg_m$, where $s \in \N$, $b_m\in \B$ and $g_m\in C_c^\infty (\mathbb{R}^n)$,
    $m=1, \dots, s$.
    There exists a set $W\subset \mathbb{R}^n$ such that $|\mathbb{R}^n\setminus W|=0$ and, for every $x\in W$,
    $$[J(S_\B g)(x)](y,t)
        =\chi _{B(x,t)}(y,t)(S_\B g)(y,t)=\sum_{m=1}^sb_m\chi _{B(x,t)}(y,t)(Sg_m)(y,t),\quad (y,t)\in \mathbb{R}^n\times (0,\infty ),$$
    is in $L^2\Big(\mathbb{R}^n\times (0 ,\infty ),\frac{dydt}{t^{n+1}};\B\Big)$ and
    \begin{align*}
        \langle L, [J(S_\B g)(x)](h)\rangle _{\B^*,\B}
            &= \sum_{m=1}^s\langle L,b_m\rangle _{\B^*,\B}\int_{\Gamma (x)}(Sg_m)(y,t)h(y,t)\frac{dydt}{t^{n+1}}\\
            &= \int_{\mathbb{R}^n\times (0,\infty )}\langle L,[J(S_\B g)(x)](y,t)\rangle  _{\B^*,\B} h(y,t)\frac{dydt}{t^{n+1}},\quad x\in W.
    \end{align*}

    For every $\ell \in \mathbb{N}$, we denote by $W_\ell$ the set in $\mathbb{R}^n$
    associated with $f_{k_\ell}$ as above. We define $\Omega _0=\Omega \cap (\cap_{\ell \in \mathbb{N}}W_\ell )$.
    We have that $|\mathbb{R}^n\setminus \Omega _0|=0$.

    Let $x\in \Omega _0$. We can write
    \begin{align*}
        \langle L, [G(x)](h)\rangle _{\B^*,\B}&=\lim_{\ell \rightarrow  \infty} \langle L, [J(S_\B f_{k_\ell })(x)](h)\rangle _{\B^*,\B}\\
        &=\lim_{\ell \rightarrow \infty }\int_{\mathbb{R}^n\times (0,\infty )}\langle L, [J(S_\B f_{k_\ell })(x)](y,t)\rangle _{\B^*,\B}h(y,t)\frac{dydt}{t^{n+1}}\\
        &=\int_{\mathbb{R}^n\times (0,\infty )}\langle L, [J(Sf)(x)](y,t)\rangle _{\B^*,\B}h(y,t)\frac{dydt}{t^{n+1}}.
    \end{align*}
    We conclude that $\langle L,  [J(Sf)(x)](\cdot, \cdot )\rangle _{\B^*,\B}\in H$, $x\in \Omega _0$.
    Then, for every $x\in \Omega _0$,
    $$G(x)
        =[J(Sf)(x)](\cdot, \cdot ),$$
    as elements of $\gamma(H,\B)$.

    Hence, $Sf=\widetilde{S_\B}f$. Thus, (\ref{2.0.1}) is proved.
\end{proof}

Before establishing the first inequality in Theorem~\ref{teo 1.1} we need the following polarization formula.

\begin{lema}\label{Lem2.3}
    Let $f\in C_c^\infty (\mathbb{R}^n)\otimes \B$ and $g\in C_c^\infty (\mathbb{R}^n)\otimes \B^*$. Then,
    $$\int_{\mathbb{R}^n}\int_{\Gamma (x)}\langle t^\beta \partial _t^\beta P_t(g)(y),t^\beta \partial _t^\beta P_t(f)(y)\rangle_{\B^*,\B}\frac{dydt}{t^{n+1}}dx
        =v_n \frac{\Gamma (2\beta )}{2^{2\beta }}\int_{\mathbb{R}^n} \langle g(y),f(y)\rangle _{\B^*,\B}dy,$$
    where $v_n$ denotes the volume of the unit ball in $\R^n$.
\end{lema}

\begin{proof}
    Suppose firstly that $f,g\in C_c^\infty (\mathbb{R}^n)$.
    According to \eqref{eqB} and by using Plancherel equality we get
    \begin{align}\label{2.9}
        &\int_{\mathbb{R}^n}\int_{\Gamma (x)}t^\beta \partial _t^\beta P_t(f)(y)t^\beta \partial _t^\beta P_t(g)(y)\frac{dydt}{t^{n+1}}dx
	=\int_{\mathbb{R}^n}\int_{\Gamma (x)}t^\beta \partial _t^\beta P_t(f)(y)\overline{t^\beta \partial _t^\beta P_t(\overline{g})(y)}\frac{dydt}{t^{n+1}}dx\nonumber\\
        & \qquad =\int_{\mathbb{R}^n}\int_{\Gamma (x)}[(|z|t)^\beta e^{-t|z|}\widehat{f}]^{\widecheck{\;}}(y) \overline{[(|z|t)^\beta e^{-t|z|}\widehat{\overline{g}}]^{\widecheck{\;}}(y)}\frac{dydt}{t^{n+1}}dx \nonumber \\
        & \qquad =v_n\int_0^\infty \int_{\mathbb{R}^n}(|z|t)^{2\beta }e^{-2t|z|}\widehat{f}(z)\overline{\widehat{\overline{g}}(z)}\frac{dzdt}{t}
         =v_n \frac{\Gamma (2\beta)}{2^{2\beta }}\int_{\mathbb{R}^n}f(y)g(y)dy.
    \end{align}
    The interchanges in the order of integration are justified because the integrals are absolutely convergent.
    In order to see this it is sufficient recall that the Littlewood-Paley-Stein function defined by
    $$g_\beta (f)(x)
        =\left(\int_0^\infty |t^\beta \partial _t^\beta P_t(f)(x)|^2\frac{dt}{t}\right)^{1/2}$$
    is bounded from $L^2(\mathbb{R}^n)$ into itself.

    From (\ref{2.9}) we easily conclude the proof of this lemma.
\end{proof}

\begin{lema}\label{Lem2.2}
    Let $\B$ be a UMD Banach space, $1<p<\infty$ and $\beta>0$.
    Then, there exists $C>0$ such that
    $$\|f\|_{L^p(\mathbb{R}^n,\B)}
            \leq C \|t^\beta \partial _t^\beta P_t(f)\|_{T_p^2(\mathbb{R}^n,\B)},\quad f \in L^p(\mathbb{R}^n,\B).$$
\end{lema}

\begin{proof}
    Let $f\in C_c^\infty (\mathbb{R}^n)\otimes \B$ and $g\in C_c^\infty (\mathbb{R}^n)\otimes \B^*$.
    Since $\B$ is a UMD Banach space, $\B^*$ has also the UMD property.
    According to Lemma~\ref{Lem2.1}, Lemma~\ref{Lem2.3} and \cite[Proposition 2.4]{HW} we can write
    \begin{align*}
        & \left|\int_{\mathbb{R}^n} \langle g(y),f(y)\rangle _{\B^*,\B}dy\right| \\
            & \qquad \leq C \int_{\mathbb{R}^n}\int_{\mathbb{R}^n\times (0,\infty )}\left|\langle \chi _{B(x,t)}(y,t)t^\beta \partial _t^\beta P_t(g)(y),\chi _{B(x,t)}(y,t)t^\beta \partial _t^\beta P_t(f)(y)\rangle _{\B^*,\B}\right|\frac{dydt}{t^{n+1}}dx\\
            & \qquad \leq C \int_{\mathbb{R}^n}\| \chi _{B(x,t)}(y,t)t^\beta \partial _t^\beta P_t(f)(y)\|_{\gamma (H,\B)}
                \| \chi _{B(x,t)}(y,t)t^\beta \partial _t^\beta P_t(g)(y)\|_{\gamma (H,\B^*)}dx\\
            & \qquad \leq C \|t^\beta \partial _t^\beta P_t(f)(y)\|_{T_p^2(\mathbb{R}^n,\B)}\|t^\beta \partial _t^\beta P_t(g)(y)\|_{T_{p'}^2(\mathbb{R}^n,\B^*)}\\
            & \qquad \leq C\|t^\beta \partial _t^\beta P_t(f)(y)\|_{T_p^2(\mathbb{R}^n,\B)}\|g\|_{L^p(\mathbb{R}^n,\B^*)}.
    \end{align*}
Then, since $C_c^\infty (\mathbb{R}^n)\otimes \B$ is a dense subspace of $L^p(\mathbb{R}^n,\B)$, by taking into account Lemma \ref{Lem2.1} we conclude the proof of this lemma.
\end{proof}

\section{Proof of Theorem \ref{Theorem 1.2}, $(i)$}\label{sec:Schrodinger}

\subsection{}
Suppose that $V\in RH_s(\R^n)$, where $s > n/2$ and
$n \geq 3$. We are going to show that, for every $f \in L^p(\R^n,\B)$,
\begin{equation*}\label{3.1}
    \|t^\beta \partial_t^\beta P_t^{\mathfrak{L}_V}(f)\|_{T_p^2(\R^n,\B)}
        \leq C\|f\|_{L^p(\R^n,\B)}.
\end{equation*}
In order to see this we cannot proceed as in the proof of Lemma~\ref{Lem2.1}.
In this Schr\"odinger case \cite[Theorem 4.8]{HNP} does not apply because
$\partial_tP_t^{\mathfrak{L}_V}(1) \not\equiv 0$. Note that, since
$\displaystyle  \lim_{t \rightarrow 0^+}P_t^{\mathfrak{L}_V}(1)(x)=1$,
$ x\in \R^n$, if $\partial_tP_t^{\mathfrak{L}_V}(1)(x)=0$, $t>0$ and
$x \in\R^n$, then $P_t^{\mathfrak{L}_V}(1)(x)=1$, $t>0$ and $x \in \R^n$, and we could infer that
$$0=(\partial_{tt}-\mathfrak{L}_V)P_t^{\mathfrak{L}_V}(1)(x)
    =-V(x)P_t^{\mathfrak{L}_V}(1)(x)=-V(x).$$
Hence, since $V \not\equiv 0$, we have that $\partial_t P_t^{\mathfrak{L}_V}(1)(x)\not\equiv 0$.

We define, for every $x \in \R^n$,
$$\rho(x)
    =\sup \left\{r>0: \frac{1}{r^{n-2}}\int_{B(x,r)} V(y) dy \leq 1\right\}.$$
The function $\rho$ (usually called critical radius) plays an
important role in the development of the harmonic analysis
associated with $\mathfrak{L}_V$ (see, for instance,
\cite{DGMTZ},
\cite{DZ1},
\cite{DZ3} and
\cite{Sh}). The main properties of
$\rho$ can be encountered in \cite[Section 1]{Sh}.

Suppose that $f =\sum_{j=1}^m b_j f_j$, where $b_j \in \B$ and $f_j \in C_c^\infty(\R^n)$, $j=1,\dots,m$.
It is clear that
$$t^\beta\partial_t^\beta P_t^{\mathfrak{L}_V}(f)(y)
    = \sum_{j=1}^m b_j t^\beta\partial_t^\beta P_t^{\mathfrak{L}_V}(f_j)(y), \quad t>0\;\;\mbox{and}\;\;y \in \R^n.$$
Moreover, for almost every $x\in \R^n$, we have that
$\chi_{\Gamma(x)}(y,t)t^\beta \partial_t^\beta P_t^{\mathfrak{L}_V}(f)(y) \in H \otimes \B$.
Indeed, let $g \in L^2(\R^n)$, we can write
\begin{align*}
    \int_{\R^n}\int_{\Gamma(x)} \left|t^\beta\partial_t^\beta P_t^{\mathfrak{L}_V}(g)(y)\right|^2\frac{dtdy}{t^{n+1}}
        &= v_n\int_{\R^n}\int_0^\infty \left|t^\beta\partial_t^\beta P_t^{\mathfrak{L}_V}(g)(y)\right|^2 \frac{dtdy}{t},
\end{align*}
where once again $v_n$ denotes the volume of the unit ball in $\R^n$.
Then, according to \cite[Theorem A]{AST}, it follows that
$$\int_{\R^n}\int_{\Gamma(x)}\left|t^\beta\partial_t^\beta P_t^{\mathfrak{L}_V}(g)(y)\right|^2\frac{dydt}{t^{n+1}}dx
    \leq C \|g\|^2_{L^2(\R^n)}.$$
Hence, for almost $x\in \R^n$,
$t^\beta \partial_t^\beta P_t^{\mathfrak{L}_V}(g)(y) \chi_{\Gamma(x)}(t,y) \in H \otimes \B$.

To simplify, we write
$$ K^{\mathfrak{L}_V}(f)(x;y,t)
    = t^\beta \partial_t^\beta P_t^{\mathfrak{L}_V}(f)(y)\chi_{\Gamma(x)}(y,t), \quad x,y\in \R^n\;\;\mbox{and}\;\; t>0.$$
We decompose the operator $K^{\mathfrak{L}_V}$ as follows:
$$ K^{\mathfrak{L}_V}(f)(x;y,t)
    = K^{\mathfrak{L}_V}_{loc}(f)(x;y,t) + K^{\mathfrak{L}_V}_{glob}(f)(x;y,t), \quad x,y\in \R^n\;\;\mbox{and}\;\; t>0,$$
where
$$ K^{\mathfrak{L}_V}_{loc}(f)(x;y,t)
    = K^{\mathfrak{L}_V}(f\chi_{B(x,\rho(x))})(x;y,t), \quad x,y\in \R^n\;\;\mbox{and}\;\; t>0.$$
The operators $K^{\mathfrak{L}_V}_{loc}$ and $K^{\mathfrak{L}_V}_{glob}$
are usually called local and global part of $K^{\mathfrak{L}_V}$, respectively.
We also consider the operators
$$ \mathbb K(f)(x;y,t)
    = t^\beta \partial^\beta_tP_t(f)(y,t)\chi_{\Gamma(x)}(y,t), \quad x,y \in \R^n\;\;\mbox{and}\;\; t>0,$$
and
$$ \mathbb K_{loc}(f)(x;y,t)
    = \mathbb{K}(f\chi_{B(x,\rho(x))})(x;y,t), \quad x,y \in \R^n\;\;\mbox{and}\;\; t>0.$$
We can write
\begin{equation}\label{3.2}
    K^{\mathfrak{L}_V}(f)(x;y,t)
        = D^{\mathfrak{L}_V}(f)(x;y,t) + \mathbb{K}_{loc}(f)(x;y,t) + K^{\mathfrak{L}_V}_{glob}(f)(x;y,t),
        \quad x,y \in \R^n\;\;\mbox{and}\;\;t>0,
\end{equation}
where
$$D^{\mathfrak{L}_V}(f)(x;y,t)
    = K^{\mathfrak{L}_V}_{loc}(f)(x;y,t)-\mathbb{K}_{loc}(f)(x;y,t), \quad x,y \in \R^n\;\;\mbox{and}\;\;t>0.$$
The three terms in the right hand side of \eqref{3.2} are
studied separately in the following lemmas.

\begin{lema}\label{Lem3.1}
    Let $\B$ be a Banach space and $1<p<\infty$. Then,
    $$\| K^{\mathfrak{L}_V}_{glob}(f) \|_{L^p(\R^n,\gamma(H,\B))}
        \leq C \|f\|_{L^p(\R^n,\B)}, \quad f \in C_c^\infty(\R^n) \otimes \B.$$
\end{lema}

\begin{proof}
    Let $f =\sum_{j=1}^m b_j f_j$, where $b_j \in \B$ and $f_j \in C_c^\infty(\R^n)$, $j=1,\dots,m$.
    Firstly, we show that for every
    $x\in \R^n$, $K^{\mathfrak{L}_V}_{glob}(f)(x;\cdot,\cdot) \in \gamma(H,\B)$.
    Fix $x\in \R^n$. According to the subordination formula, we have that
    $$ P_t^{\mathcal{L}_V}(y,z)
        = \frac{t}{\sqrt{4\pi}}\int_0^\infty \frac{e^{-t^2/4u}}{u^{3/2}} W_u^{\mathcal{L}_V}(y,z)du,
        \quad y,z \in \R^n\;\;\mbox{and}\;\;t>0.$$
    By \cite[Lemma 4]{BCCFR1}
    \begin{equation}\label{3.3}
        |\partial_t^\beta [te^{-t^2/4u}]|
            \leq C e^{-t^2/8u}u^{(1-\beta)/2}, \quad t,u \in (0,\infty).
    \end{equation}
    Also, by tacking into account the Feynman-Kac formula (\cite[(2.2)]{DGMTZ}) and (\ref{3.3}) we can interchange
    the order of integration and differentiate under the integral sign to get
    $$ \partial_t^\beta P_t^{\mathcal{L}_V}(y,z)
        =\frac{1}{\sqrt{4\pi}}\int_0^\infty \frac{\partial_t^\beta[te^{-t^2/4u}]}{u^{3/2}}W_u^{\mathcal{L}_V}(y,z)du, \quad x,y\in\R^n\;\;\mbox{and}\;\;t>0.$$
    Then, by using \cite[(2.3)]{DGMTZ} and (\ref{3.3}) we obtain
    \begin{align}\label{3.4}
        |t^\beta\partial_t^\beta P_t^{\mathcal{L}_V}(y,z)|
            &\leq C t^\beta \rho(y) \int_0^\infty e^{-t^2/8u} u^{-(3+\beta+n)/2}e^{-c|y-z|^2/u}du\nonumber\\
            &\leq C\frac{ t^\beta\rho(y)}{(t+|y-z|)^{n+\beta+1}}, \quad x,y\in\R^n\;\;\mbox{and}\;\;t>0.
    \end{align}
    Estimation (\ref{3.4}) justifies the differentiation under the integral sign, and we get
    \begin{align*}
        &\left\|K^{\mathfrak{L}_V}_{glob}(f)(x;y,t)\right\|_\B\\
        & \qquad \leq C \chi_{\Gamma(x)}(y,t) \int_{\R^n \setminus B(x,\rho(x))}\frac{t^\beta\rho(y)}{(t+|y-z|)^{n+\beta+1}}\|f(z)\|_{\B} dz  \\
        & \qquad \leq C t^\beta \rho(y) \chi_{\Gamma(x)}(y,t)  \|f\|_{L^p(\R^n,\B)}\left(\int_{\R^n \setminus B(x,\rho(x))}\frac{dz}{(t+2|x-y|+|y-z|)^{(n+\beta+1)p'}}\right)^{1/p'}\\
        & \qquad \leq C t^\beta \rho(y) \chi_{\Gamma(x)}(y,t)  \|f\|_{L^p(\R^n,\B)}\left(\int_{\R^n \setminus B(x,\rho(x))}\frac{dz}{(t+|x-y|+|z-x|)^{(n+\beta+1)p'}}\right)^{1/p'}\\
        & \qquad \leq C t^\beta \rho(y)\chi_{\Gamma(x)}(y,t) \|f\|_{L^p(\R^n,\B)} \left(\int_{\rho(x)}^\infty \frac{r^{n-1}dr}{(t+|x-y|+r)^{(n+\beta+1)p'}}\right)^{1/p'} \\
        & \qquad \leq C \frac{ t^\beta \rho(y)\chi_{\Gamma(x)}(y,t)}{(t+|x-y|+\rho(x))^{n+\beta+1-n/p' }}\|f\|_{L^p(\R^n,\B)},
        \quad x,y\in \R^n\;\;\mbox{and}\;\; t>0.
    \end{align*}

    According to \cite[Proposition 1]{DGMTZ} we deduce that
    \begin{equation}\label{3.5}
        \rho(y)
            \leq C\rho(x) \left(1+\frac{t}{\rho(x)}\right)^{\gamma}, \quad |x-y| <t, \quad x,y\in \R^n\;\;\mbox{and}\;\; t>0,
    \end{equation}
    for some $ 1/2 \leq \gamma <1$.  Then,
    \begin{align*}
        &\int_{\R^n \times (0,\infty)}\|K^{\mathfrak{L}_V}_{glob}(f)(x;y,t)\|_\B^2\frac{dtdy}{t^{n+1}} \\
        & \qquad \leq C \rho(x)^2\|f\|^2_{L^p(\R^n,\B)}\int_0^\infty \int_{|x-y|<t} \frac{t^{2\beta-n-1}}{(t+|x-y|+\rho(x))^{2n+2\beta+2-2n/p'}}\left(1+\frac{t}{\rho(x)}\right)^{2\gamma}dydt \\
        & \qquad \leq C \rho(x)^2\|f\|^2_{L^p(\R^n,\B)}\int_0^\infty  \frac{t^{2\beta-1}}{(t+\rho(x))^{2n+2\beta+2-2n/p'}}\left(1+\frac{t}{\rho(x)}\right)^{2\gamma}dt \\
        & \qquad \leq C \rho(x)^{-2n/p} \|f\|^2_{L^p(\R^n,\B)}, \quad    x \in \R^n.
    \end{align*}
    Here we have used that $0 < \gamma <1$. Since $0<\rho(z)<\infty$,
    $z\in \R^n$ (\cite[p. 519]{Sh}), we conclude that
    $K^{\mathfrak{L}_V}_{glob}(f)(x;\cdot,\cdot)\in L^2\left(\R^n\times(0,\infty),\frac{dydt}{t^{n+1}};\B\right)$, $x\in \mathbb{R}^n$.

    To simplify we write $G(x)=K^{\mathfrak{L}_V}_{glob}(f)(x;\cdot,\cdot)$ and $G_j(x)=K^{\mathfrak{L}_V}_{glob}(f_j)(x;\cdot,\cdot)$, $x\in \R^n$
    and $j=1,\ldots,m$.
    We denote by $T_{G(x)}$ the linear and bounded operator  from
    $H$ into $\B$ defined by
    $$T_{G(x)}(h)
        = \int_{\R^n\times (0,\infty)}K^{\mathfrak{L}_V}_{glob}(f)(x;y,t)h(y,t)\frac{dydt}{t^{n+1}}, \quad h \in H.$$
    For every $j=1,\ldots,m$, $T_{G_j}$ is defined in a similar way.

    For every $S \in \B^*$, we have that
    $$ \langle S,T_{G(x)}(h)\rangle_{\B^*,\B}
        = \int_{\R^n\times (0,\infty )} \langle S, K^{\mathfrak{L}_V}_{glob}(f)(x;y,t)\rangle_{\B^*,\B} h(y,t) \frac{dydt}{t^{n+1}}, \quad h \in H.$$
    Moreover, it is clear that
    $$ T_{G(x)}
        = \sum_{j=1}^m b_jT_{G_j(x)}.$$
    Then, for every $h\in H$ and $S \in \B^*$ the function
    $$\begin{array}{lll}
        \R^n & \longrightarrow & \mathbb C \\
        x & \longmapsto &
            \displaystyle \langle S, T_{G(x)}(h)\rangle _{\B^*,\B}=\sum_{j=1}^n \langle  S,b_j\rangle_{\B^*,\B}T_{G_j(x)}(h)
    \end{array}$$
    is measurable. According to \cite[Lemma 2.5]{NVW} and Pettis measurability theorem
    (\cite[Proposition 2.1]{NVW}) the function
    $$\begin{array}{lll}
        \R^n & \longrightarrow & \gamma(H,\B)  \\
        x & \longmapsto & G(x) (\equiv T_{G(x)})
    \end{array}
    $$
    is strongly measurable.

    Suppose that $\{h_j\}_{j=1}^\infty $ is an orthonormal basis in $H$ and $\{\gamma_j\}_{j=1}^\infty$ is a sequence
    of independent standard normal variables on a probability space $(\Omega,\mathfrak F, P)$.
    We have that
\begin{align*}
\|G(x)\|_{\gamma (H,\mathbb{B})}=&\left(\mathbb{E}\Big\|\sum_{j=1}^\infty \gamma_jT_{G(x)}(h_j)\Big\|_\mathbb{B}^2\right)^{1/2}\\
&\hspace{-2cm}=\left(\mathbb{E}\left\|\sum_{j=1}^\infty \gamma _j\int_{\mathbb{R}^n\setminus B(x,\rho (x))}f(z)\int_{\mathbb{R}^n\times (0,\infty)}\chi _{\Gamma (x)}(y,t)t^\beta \partial_t^\beta P_t^{\mathfrak{L}_V}(y,z)h_j(y,t)\frac{dydt}{t^{n+1}}dz\right\|_\mathbb{B}^2\right)^{1/2},\quad x\in \mathbb{R}^n.
\end{align*}

The interchange of the order of integration is justified because by proceeding as above we can show that the
    integrals are norm convergent.

    Then, from (\ref{3.4}) and (\ref{3.5}) it follows that
    \begin{align*}
        \|G(x)\|_{\Rad}
        \leq & \int_{\R^n \setminus B(x,\rho(x))} \|f(z)\|_{\B}
            \Big(\mathbb{E} \Big|\sum_{j=1}^\infty\gamma_j
                \int_{\Gamma(x)}t^\beta\partial^\beta_t P_t^{\mathfrak L_V}(y,z)h_j(y,t)\frac{dydt}{t^{n+1}}\Big|^2\Big)^{1/2}dz \\
        \leq & C\int_{\R^n \setminus B(x,\rho(x))}  \|f(z)\|_{\B}
                \Big\|t^\beta\partial^\beta_tP_t^{\mathfrak L_V}(y,z)\chi_{\Gamma(x)}(y,t)\Big\|_{H}dz \\
        \leq & C \int_{\R^n \setminus B(x,\rho(x))} \|f(z)\|_{\B}
            \Big(\int_{\Gamma(x)} \frac{\rho(x)^2 \Big(1+t/\rho(x)\Big)^{2\gamma}t^{2\beta-n-1}}{(t+|y-z|)^{2(n+\beta+1)}}dtdy\Big)^{1/2}dz\\
        \leq & C \rho(x) \int_{\R^n \setminus B(x,\rho(x))} \|f(z)\|_{\B}
            \Big(\int_{\Gamma(x)} \frac{\Big(1+t/\rho(x)\Big)^{2\gamma}t^{2\beta-n-1}}{(t+2|x-y|+|y-z|)^{2(n+\beta+1)}}dtdy\Big)^{1/2}dz\\
        \leq & C \rho(x)\int_{\R^n \setminus B(x,\rho(x))} \|f(z)\|_{\B}
            \Big(\int_{0}^\infty \frac{\Big(1+t/\rho(x)\Big)^{2\gamma}t^{2\beta-1}}{(t+|x-z|)^{2(n+\beta+1)}}dt\Big)^{1/2}dz\\
        \leq & C\rho(x) \int_{\R^n \setminus B(x,\rho(x))} \|f(z)\|_{\B}
            \Big(\int_0^{\rho(x)} \frac{t^{2\beta-1}}{(t+|x-z|)^{2(n+\beta+1)}}dt \\
        & + \int_{\rho(x)}^\infty \frac{t^{2\gamma+ 2\beta-1}}{(t+|x-z|)^{2(n+\beta+1)}}\frac{dt}{\rho(x)^{2\gamma}}\Big)^{1/2}dz \\
        \leq & C \rho(x)\int_{\R^n \setminus B(x,\rho(x))}\|f(z)\|_{\B}
            \Big(\frac{\rho(x)^{2\beta}}{|x-z|^{2(n+\beta+1)}}+ \frac{1}{\rho(x)^{2\gamma}(\rho(x)+|x-z|)^{2n+2-2\gamma}}\Big)^{1/2}dz\\
        \leq & C \Big(\rho(x)^{\beta+1} \int_{\R^n \setminus B(x,\rho(x))} \frac{\|f(z)\|_{\B}}{|x-z|^{n+\beta+1}} dz
            + \rho(x)^{1-\gamma}\int_{\R^n \setminus B(x,\rho(x))} \frac{\|f(z)\|_{\B}}{|x-z|^{n+1-\gamma}} dz \Big)\\
        \leq & C \Big(\rho(x)^{\beta+1}\sum_{k=0}^\infty\int_{2^k\rho(x)\leq |x-z|<2^{k+1}\rho(x)} \frac{\|f(z)\|_{\B}}{|x-z|^{n+\beta+1}} dz\\
        & + \rho(x)^{1-\gamma} \sum_{k=0}^\infty\int_{2^k\rho(x)\leq |x-z|<2^{k+1}\rho(x)} \frac{\|f(z)\|_{\B}}{|x-z|^{n+1-\gamma}}dz\Big)\\
        \leq & C \Big(\sum_{k=0}^\infty \frac{\rho(x)^{\beta+1}}{(2^k\rho(x))^{n+\beta+1}} \int_{B(x,2^{k+1}\rho(x))} \|f(z)\|_{\B} dz\\
        & + \sum_{k=0}^\infty \frac{\rho(x)^{1-\gamma}}{(2^k\rho(x))^{n+1-\gamma}}\int_{B(x,2^{k+1}\rho(x))} \|f(z)\|_{\B} dz \Big)\\
        \leq & C \mathcal{M}(\|f\|_{\B})(x), \quad\; x\in \R^n,
    \end{align*}
    where $\mathcal M$ denotes the Hardy-Littlewood maximal operator.
    Note that we have used that $1/2 \leq \gamma < 1$.

    As it is well-known the maximal operator $\mathcal M$ is bounded from $L^p(\R^n)$ into itself. Then
    \begin{equation*}\label{M1}
        \|K^{\mathfrak{L}_V}_{glob}(f)\|_{L^p(\R^n,\gamma(H,\B))}
            \leq C\|f\|_{L^p(\R^n,\B)},
    \end{equation*}
    where $C>0$ does not depend on $f$.
\end{proof}

\begin{lema}\label{Lem3.2}
    Let $\B$ be a Banach space and $1<p<\infty$. Then,
    $$\| D^{\mathfrak{L}_V}(f) \|_{L^p(\R^n,\gamma(H,\B))}
        \leq C \|f\|_{L^p(\R^n,\B)}, \quad f \in C_c^\infty(\R^n) \otimes \B.$$
\end{lema}

\begin{proof}
    Let $f =\sum_{j=1}^m b_j f_j$, where $b_j \in \B$ and $f_j \in C_c^\infty(\R^n)$, $j=1,\dots,m$.
    We can write
    $$ \partial_t^\beta(P_t^{\mathcal{L}_V}(y,z)-P_t(y,z))
        =\frac{1}{\sqrt{4\pi}}\int_0^\infty  \frac{\partial_t^\beta[te^{-t^2/4u}]}{u^{3/2}}\left(W_u^{\mathcal L_V}(y,z)-W_u(y-z)\right)du,$$
    where
    $$ W_u(y)
        = \frac{e^{-|y|^2/4u}}{(4\pi u)^{n/2}}, \quad y\in \R^n\;\;\mbox{and}\;\;u>0.$$
    According to Kato-Trotter's formula \cite[(2.10)]{DZ2} we have that
    \begin{align*}
        W_u(y-z)- W_u^{\mathcal L_V}(y,z)
            =& \int_0^u\int_{\R^n} W_s(y-v)V(v) W_{u-s}^{\mathcal L_V}(v,z)dvds\\
            =& \int_0^{u/2}\int_{\R^n} W_s(y-v)V(v) W_{u-s}^{\mathcal L_V}(v,z)dvds\\
            &+ \int_0^{u/2}\int_{\R^n} W_{u-s}(y-v)V(v) W_u^{\mathcal L_V}(v,z)dvds, \quad y,z\in \R^n\;\;\mbox{and}\;\;u>0.
    \end{align*}
    From \cite[(2.2)]{DGMTZ} and \cite[(2.8)]{DGMTZ} we infer that
    \begin{align*}
        & \left|\int_0^{u/2}\int_{\R^n} W_s(y-v)V(v) W_{u-s}^{\mathcal L_V}(v,z)dvds\right|
            \leq \int_0^{u/2}\int_{\R^n} \frac{e^{-c|y-v|^2/s}}{s^{n/2}}V(v)\frac{e^{-c|v-z|^2/u}}{(u-s)^{n/2}} dvds \\
        & \qquad \leq C\int_0^{u/2}\int_{\R^n} \frac{e^{-c(|y-v|^2+|v-z|^2)/u}}{(su)^{n/2}}V(v)e^{-c|y-v|^2/s} dvds \\
        & \qquad \leq C  \frac{e^{-c|y-z|^2/u}}{u^{n/2}}\int_0^{u/2}\int_{\R^n}V(v)\frac{e^{-c|y-v|^2/s}}{s^{n/2}} dvds
        \leq C\frac{e^{-c|y-z|^2/u}}{u^{n/2}}\int_0^{u/2}\frac{1}{s}\left(\frac{s}{\rho(y)^2}\right)^\delta ds \\
        & \qquad = C\frac{e^{-c|y-z|^2/u}}{u^{-\delta + n/2}\rho(y)^{2\delta}}, \quad0<u<\rho(y)^2, \quad y,z\in \R^n,
    \end{align*}
    for a certain $\delta >0$.
    Also, we get
    $$ \left|\int_0^{u/2}\int_{\R^n} W_{u-s}(y-v)V(v) W_s^{\mathcal L_V}(v,z)dvds\right|
        \leq C\frac{e^{-c|y-z|^2/u}}{u^{-\delta + n/2}\rho(z)^{2\delta}}, \quad0<u<\rho(z)^2, \quad y,z\in \R^n.$$
       Then, according to \cite[Lemma 1.4, (a)]{Sh} if $|x-z|<\rho(x)$ and $|y-z| < 2\rho(x)$, then
    $\rho(y) \sim \rho(x) \sim \rho(z)$. Hence, we conclude that
    \begin{equation}\label{3.6}
        |W_u(y-z) -W_u^{\mathcal L_V}(y,z)|
            \leq C\frac{e^{-c|y-z|^2/u}}{u^{-\delta + n/2}\rho(z)^{2\delta}},
    \end{equation}
    for $0<u<\rho(z)^2$ and  $y,z\in \R^n$ such that $|y-z| \leq 2\rho(x)$ and $|x-z| < \rho(x)$.
    Estimates (\ref{3.3}) and (\ref{3.6}) lead to
    \begin{equation}\label{3.7}
        \Big|\int_0^{\rho(z)^2} \frac{\partial_t^\beta[te^{-t^2/4u}]}{u^{3/2}}
            \Big(W_u^{\mathcal L_V}(y,z)-W_u(y-z)\Big)du \Big|
            \leq \frac{C}{\rho(x)^{2\delta}} \int_0^{\rho(z)^2} \frac{e^{-c(t^2+|y-z|^2)/u}}{u^{1-\delta+(n+\beta)/2}}du,
    \end{equation}
    for $y,z\in \R^n$ such that $|y-z|\leq 2\rho(x)$ and $|x-z| < \rho(x)$.

    On the other hand, \cite[(2.2)]{DGMTZ} and (\ref{3.3}) imply that
    \begin{equation}\label{3.9}
        \Big|\int^\infty_{\rho(z)^2} \frac{\partial_t^\beta[te^{-t^2/4u}]}{u^{3/2}}
            \Big(W_u^{\mathcal L_V}(y,z)-W_u(y-z)\Big)du \Big|
            \leq C\int^\infty_{\rho(z)^2} \frac{e^{-c(t^2+|y-z|^2)/u}}{u^{1+(n+\beta)/2}}du, \quad y,z\in \R^n,\;t>0,
    \end{equation}
    and also
    \begin{equation}\label{3.10}
        \Big|\int_0^{\rho(z)^2} \frac{\partial_t^\beta[te^{-t^2/4u}]}{u^{3/2}}
            \Big(W_u^{\mathcal L_V}(y,z)-W_u(y-z)\Big)du \Big|
            \leq C\frac{1}{(t+|y-z|)^{n+\beta}}, \quad y,z\in \R^n\text{ and } t>0.
    \end{equation}

    We are going to see that $ D^{\mathfrak{L}_V}(f)(x;\cdot,\cdot) \in H$, for every $x\in \R^n$.
    Fix $x \in \R^n$. We can write
    \begin{align*}
        D^{\mathfrak{L}_V}(f)(x;y,t)
            = & \chi_{\Gamma(x)}(y,t)\int_{B(x,\rho(x))}f(z)\frac{t^\beta}{\sqrt{4\pi}}\int_0^{\rho(z)^2}
                \frac{\partial_t^\beta[te^{-t^2/4u}]}{u^{3/2}} \Big(W_u^{\mathcal L_V}(y,z)-W_u(y-z)\Big)du dz\\
              & + \chi_{\Gamma(x)}(y,t)\int_{B(x,\rho(x))}f(z)\frac{t^\beta}{\sqrt{4\pi}}\int^\infty_{\rho(z)^2}
                \frac{\partial_t^\beta[te^{-t^2/4u}]}{u^{3/2}} \Big(W_u^{\mathcal L_V}(y,z)-W_u(y-z)\Big)du dz\\
            = & D^{\mathfrak{L}_V}_1(f)(x;y,t) + D^{\mathfrak{L}_V}_2(f)(x;y,t), \quad y\in \R^n\;\; \mbox{and}\;\;t>0.
    \end{align*}
    Minkowski's inequality and (\ref{3.9}) leads to
    \begin{align*}
        &\|D^{\mathfrak{L}_V}_2(f)(x;\cdot,\cdot)\|_{L^2\Big( \R^n \times (0,\infty), \frac{dydt}{t^{n+1}};\B\Big)}\\
        & \qquad \leq C \int_{B(x,\rho(x))}\|f(z)\|_{\B}\Big\|t^\beta \int^\infty_{\rho(z)^2}
            \frac{\partial_t^\beta[te^{-t^2/4u}]}{u^{3/2}} \Big(W_u^{\mathcal L_V}(y,z)-W_u(y-z)\Big)du\Big\|_{H} dz\\
        &\qquad \leq C\int_{B(x,\rho(x))} \|f(z)\|_{\B}  \int_{\rho(z)^2}^\infty\frac{1}{u^{1 + (n+\beta)/2}}
            \Big(\int_0^\infty\int_{|x-y|<t}e^{-c(t^2+|y-z|^2)/u}\frac{dydt}{t^{-2\beta+n+1}}\Big)^{1/2}dudz\\
        &\qquad \leq C\int_{B(x,\rho(x))}\|f(z)\|_{\B}\int_{\rho(z)^2}^\infty\frac{1}{u^{1+(n+\beta)/2}}
            \Big(\int_0^\infty e^{-c t^2 / u}\frac{dt}{t^{1-2\beta}}\Big)^{1/2}dudz\\
        &\qquad \leq C\int_{B(x,\rho(x))} \|f(z)\|_{\B}  \int_{\rho(z)^2}^\infty\frac{1}{u^{1+n/2}}dudz
        \leq \frac{C}{\rho(x)^n}\int_{B(x,\rho(x))} \|f(z)\|_{\B} dz
        \leq C \mathcal M(\|f\|_{\B})(x).
    \end{align*}
    We have taken into account that $\rho(z) \sim \rho(x)$ because $|x-z| < \rho(x)$.

    We now decompose $D^{\mathfrak{L}_V}_1(f)$ as follows
    $$ D^{\mathfrak{L}_V}_1(f)
        =D^{\mathfrak{L}_V}_{1,1}(f)+D^{\mathfrak{L}_V}_{1,2}(f),$$
    where
    $$D^{\mathfrak{L}_V}_{1,1}(f)(x;y,t)=D^{\mathfrak{L}_V}_{1}(f\chi_{B(y,2\rho(x))})(x;y,t), \quad y\in \R^n, \ t>0.$$
    By using again the Minkowski's inequality and (\ref{3.10}) we get
    \begin{align*}
        & \|D^{\mathfrak{L}_V}_{1,2}(f)(x;\cdot,\cdot)\|_{L^2\Big( \R^n \times (0,\infty), \frac{dydt}{t^{n+1}};\B\Big)} \\
        & \qquad \leq C\int_{B(x,\rho(x))} \|f(z)\|_{\B} \left(\int_0^\infty\int_{|x-y|<t}\frac{t^{2\beta-n-1}}{(t+|x-y|)^{2n+2\beta}}\chi_{\R^n \setminus B(y,2\rho(x))}(z)dydt\right)^{1/2}dz\\
        & \qquad \leq C\int_{B(x,\rho(x))}\|f(z)\|_{\B}\left(\int_0^\infty \frac{t^{2\beta-1}}{(t+\rho(x))^{2n+2\beta}}dt\right)^{1/2}dz
            \leq \frac{C}{\rho(x)^{n}}\int_{B(x,\rho(x))} \|f(z)\|_{\B} dz \\
        & \qquad \leq C\mathcal M(\|f\|_{\B})(x).
    \end{align*}
    Now, estimations \eqref{3.3} and (\ref{3.7}) imply that
    \begin{align*}
        & \|D^{\mathfrak{L}_V}_{1,1}(f)(x;\cdot,\cdot)\|_{L^2\Big( \R^n \times (0,\infty), \frac{dydt}{t^{n+1}};\B\Big)}\\
        & \qquad \leq C\bigg(\int_0^\infty\int_{|x-y|<t}\bigg(\int_{B(x,\rho(x)) \cap B(y,2\rho(x))} \|f(z)\|_{\B} \\
        & \qquad \qquad \times \Big|t^\beta\int_0^{\rho(z)^2} \frac{\partial_t^\beta[te^{-t^2/4u}]}{u^{3/2}}
            \Big(W_u^{\mathcal L_V}(y,z)-W_u(y-z)\Big)du\Big|dz\bigg)^2 \frac{dydt}{t^{n+1}}\bigg)^{1/2} \\
        & \qquad \leq C\bigg(\int_0^\infty\int_{|x-y|<t}\bigg(\int_{\R^n} \|f(z)\|_{\B}
            \frac{t^\beta}{\rho(x)^{2\delta}}\int_0^{c\rho(x)^2}
            \frac{e^{-c(t^2+|y-z|^2)/u}}{u^{1-\delta+(n+\beta)/2}}du dz\bigg)^2 \frac{dydt}{t^{n+1}}\bigg)^{1/2} \\
        & \qquad \leq C\bigg(\int_0^\infty\int_{|x-y|<t}\bigg(\int_{\R^n} \|f(z)\|_{\B}
            \frac{t^\beta}{\rho(x)^{2\delta}}\int_0^{c\rho(x)^2}
            \frac{e^{-c(t^2+|x-y|^2+|y-z|^2)/u}}{u^{1-\delta+(n+\beta)/2}}du dz\bigg)^2 \frac{dydt}{t^{n+1}}\bigg)^{1/2} \\
        & \qquad \leq C\bigg(\int_0^\infty\int_{|x-y|<t}\bigg(\int_{\R^n} \|f(z)\|_{\B}
            \frac{t^\beta}{\rho(x)^{2\delta}}\int_0^{c\rho(x)^2}
            \frac{e^{-c(t^2+|x-z|^2)/u}}{u^{1-\delta+(n+\beta)/2}}du dz\bigg)^2 \frac{dydt}{t^{n+1}}\bigg)^{1/2} \\
        & \qquad \leq C\bigg(\int_0^\infty\int_{|x-y|<t} \frac{t^{2\beta}}{\rho(x)^{4\delta}}
            \bigg(\int_0^{c\rho(x)^2} \frac{e^{-c t^2 / u}}{u^{1-\delta+(n+\beta)/2}}
            \int_{\R^n} e^{-c|x-z|^2/u} \|f(z)\|_{\B} dz du\bigg)^2 \frac{dydt}{t^{n+1}}\bigg)^{1/2} \\
        & \qquad \leq C\sup_{u>0}\int_{\R^n}\frac{e^{-c|x-z|^2/u}}{u^{n/2}} \|f(z)\|_{\B} dz
            \int_0^{c\rho(x)^2}\frac{u^{\delta-1-\beta/2}}{\rho(x)^{2\delta}}
            \bigg(\int_0^\infty\int_{|x-y|<t}t^{2\beta-n-1}e^{-c t^2 / u}dydt\bigg)^{1/2}du \\
        & \qquad \leq C\sup_{u>0}\int_{\R^n}\frac{e^{-c|x-z|^2/u}}{u^{n/2}} \|f(z)\|_{\B} dz
            \int_0^{c\rho(x)^2}\frac{u^{\delta-1-\beta/2}}{\rho(x)^{2\delta}}
            \bigg(\int_0^\infty t^{2\beta-1}e^{-c t^2 / u}dt\bigg)^{1/2}du \\
        & \qquad \leq C\sup_{u>0}\int_{\R^n}\frac{e^{-c|x-z|^2/u}}{u^{n/2}} \|f(z)\|_{\B} dz
            \int_0^{c\rho(x)^2}\frac{u^{\delta-1}}{\rho(x)^{2\delta}}du
        \leq C W_*(\|f\|_{\B})(x),
    \end{align*}
    where $W_*$ represents the maximal operator associated with the heat semigroup $\{W_t\}_{t>0}$ defined by
    $$ W_*(g)
        = \sup_{t>0}|W_t(g)|, \quad g \in L^q(\R^n), \quad 1 \leq q \leq \infty.$$
    We conclude that
    $$\|D^{\mathfrak{L}_V}(f)(x;\cdot,\cdot)\|_{L^2\Big( \R^n \times (0,\infty), \frac{dydt}{t^{n+1}};\B\Big)}
        \leq C\Big(W_*(\|f\|_{\B})(x)+ \mathcal M(\|f\|_{\B})(x)\Big).$$
    Also from the above estimations we get
    $$ \|D^{\mathfrak{L}_V}(f)(x;\cdot,\cdot)\|_{L^2\Big( \R^n \times (0,\infty), \frac{dydt}{t^{n+1}};\B\Big)}
        \leq C\|f\|_{L^\infty(\R^n,\B)}, \quad x\in\R^n.$$
    We can now proceed as in the study of $K^{\mathfrak{L}_V}_{glob}$ to obtain that
    \begin{equation*}\label{M2}
        \|D^{\mathfrak{L}_V}(f)\|_{L^p(\R^n,\gamma(H,\B))}
            \leq C\|f\|_{L^p(\R^n,\B)},
    \end{equation*}
    where $C$ does not depend on $f$, because $\mathcal M$ and $W_*$ are bounded operators from $L^p(\R^n)$ into itself.
\end{proof}

\begin{lema}\label{Lem3.3}
    Let $\B$ be a UMD Banach space and $1<p<\infty$. Then,
    $$\| \mathbb{K}_{loc}(f) \|_{L^p(\R^n,\gamma(H,\B))}
        \leq C \|f\|_{L^p(\R^n,\B)}, \quad f \in C_c^\infty(\R^n) \otimes \B.$$
\end{lema}

\begin{proof}
    Let $f =\sum_{j=1}^m b_j f_j$, where $b_j \in \B$ and $f_j \in C_c^\infty(\R^n)$, $j=1,\dots,m$.
    We are going to use the ideas developed in the proof of \cite[Theorem 3.7]{AST}.

    Our first goal is to see that
    $\mathbb K_{loc}(f)(x;\cdot,\cdot)\in L^2\left(\R^n\times (0,\infty),\frac{dydt}{t^{n+1}};\B\right)$,
    for every $x \in \R^n$.
    According to \cite[Lemma 2.3]{DZ1} we consider a sequence $(x_k)_{k=1}^\infty$ in $\R^n$
    such that if $Q_k=B(x_k,\rho(x_k))$, $k \in \mathbb N$, we have that
    \begin{itemize}
        \item[$(i)$] $\displaystyle \bigcup_{k=1}^\infty Q_k=\R^n$,
        \item[$(ii)$] There exists $N=N(\rho)$ such that, for every $k \in \mathbb N$, card$\{j\in \mathbb N: Q_j^{**}\cap Q_k^{**} \not= \emptyset\} \leq N$, where $Q_\ell^{**}= B(x_\ell,4\rho(x_\ell))$, $\ell \in \mathbb N$.
    \end{itemize}
    Fix $x \in \R^n$. We define the operators
    $$ S(f)(x;y,t)
        = \sum_{k=1}^\infty \chi_{Q_k}(x)\mathbb K(f\chi_{Q_k^*})(x;y,t), \quad y \in \R^n\;\;\mbox{and}\;\; t>0,$$
    where $Q_k^*=B(x_k, 2\rho(x_k))$, $k \in \mathbb N$, and
    $$ \mathbb S(f)(x;y,t)
        = \sum_{k=1}^\infty\chi_{Q_k}(x)\mathbb K_{loc}(f)(x;y,t)- S(f)(x;y,t), \quad y \in \R^n\;\;\mbox{and}\;\; t>0.$$
    We can write for every $y \in \R^n$ and $t>0$,
    $$\mathbb S(f)(x;y,t)
        = \chi_{\Gamma(x)}(y,t)\int_{\R^n}\sum_{k=1}^\infty \chi_{Q_k}(x)t^\beta\partial_t^\beta P_t(y-z)
            \Big(\chi_{B(x,\rho(x))}(z) - \chi_{Q_k^*}(z)\Big)f(z)dz.$$
    According to \cite[Lemma 1.4, (a)]{Sh} we deduce that, if
    $\chi_{Q_k}(x)\left(\chi_{B(x,\rho(x))}(z) - \chi_{Q_k^*}(z)\right)\not=0$, for some  $k \in \mathbb N$;
    then $\frac{1}{C_1}\rho(x)\leq |x-z| \leq C_1\rho(x)$, for a certain $C_1>0$.
    By (\ref{2.3}) it follows that
    \begin{align*}
        & \Big\|t^\beta\partial_t^\beta P_t(y-z)\chi_{\Gamma(x)}(y,t)\Big\|_{H}
            \leq C\left(\int_0^\infty\int_{|x-y|<t} \frac{t^{2\beta-n-1}}{(t+|y-z|)^{2(n+\beta)}}dydt\right)^{1/2} \\
        & \qquad \leq C\left(\int_0^\infty\int_{|x-y|<t} \frac{t^{2\beta-n-1}}{(t+|x-y|+|y-z|)^{2(n+\beta)}}dydt\right)^{1/2} \\
        & \qquad \leq C\left(\int_0^\infty\int_{|x-y|<t} \frac{t^{2\beta-n-1}}{(t+|x-z|)^{2(n+\beta)}}dydt\right)^{1/2}
            \leq C\left(\int_0^\infty \frac{t^{2\beta-1}}{(t+|x-z|)^{2(n+\beta)}}dt\right)^{1/2} \\
        & \qquad \leq \frac{C}{|x-z|^n}, \quad x,z\in \R^n.
    \end{align*}
    Minkowski's inequality and the property $(ii)$ of the sequence $\{Q_k^{**}\}^\infty_{k=1}$ lead to
    \begin{align*}\label{3.11}
        \|\mathbb S(f)(x;\cdot,\cdot)\|_{L^2\left(\R^n\times(0,\infty),\frac{dydt}{t^{n+1}},\B\right)}
            & \leq C\int_{\rho(x)/C_1\leq |x-z| \leq C_1\rho(x)}  \frac{\|f(z)\|_{\B}}{|x-z|^n}dz \nonumber \\
            & \leq C\frac{1}{\rho(x)^n}\int_{|x-z|\leq C_1\rho(x)} \|f(z)\|_{\B} dz
                \leq C\|f\|_{L^\infty(\R^n,\B)}.
    \end{align*}
    Note that by virtue of ($ii$), $C$ does not depend on $x\in \R^n$.

    By proceeding as in the study of $K_{glob}^{\mathfrak{L}_V}$ we conclude that
    $$\|\mathbb S(f)\|_{L^p(\R^n,\gamma(H,\B))}
        \leq C\|\mathcal M(\|f\|_{\B})\|_{L^p(\R^n)}
        \leq C\|f\|_{L^p(\R^n)}.$$
    On the other hand, according to Lemma~\ref{Lem2.1} and by taking into account
    the properties of the sequence $(x_k)_{k\in\mathbb N}$ we get
    \begin{align*}
        \|Sf\|_{L^p(\R^n,\gamma(H,\B))}
            & \leq C \bigg(\sum_{k=1}^\infty \|\chi_{Q_k}\mathbb K(f\chi_{Q_k^*})\|^p_{L^p(\R^n,\gamma(H,\B))}\bigg)^{1/p}
                \leq C\bigg(\sum_{k=1}^\infty\|f\chi_{Q_k^*}\|^p_{L^p(\R^n,\B)}\bigg)^{1/p} \\
            & \leq C\|f\|_{L^p(\R^n,\B)}.
    \end{align*}
    Also, we have that
    $$ \|\mathbb K_{loc}(f)\|_{L^p(\R^n,\gamma(H,\B))}
        \sim \bigg\|\sum_{k=1}^\infty\chi_{Q_k}\mathbb K_{loc}(f)\bigg\|_{L^p(\R^n,\gamma(H,\B))}.$$
    Then, we conclude that
    \begin{equation*}\label{M3}
        \|\mathbb K_{loc}(f)\|_{L^p(\R^n,\gamma(H,\B))}
            \leq C\|f\|_{L^p(\R^n,\B)},
    \end{equation*}
    where $C$ does not depend on $f$.
\end{proof}

By combining Lemmas \ref{Lem3.1}, \ref{Lem3.2} and \ref{Lem3.3} we obtain
$$ \|K^{\mathfrak{L}_V}(f)\|_{L^p(\R^n,\gamma(H,\B))}
    \leq C\|f\|_{L^p(\R^n,\B)}, \quad f \in C_c^\infty(\R^n) \otimes \B,$$
or, in other words,
\begin{equation*}\label{M4}
    \Big\|t^\beta \partial_t^\beta P_t^{\mathfrak{L}_V}(f)\Big\|_{T^2_p(\R^n,\B)}
        \leq C\|f\|_{L^p(\mathbb  R^n,\B)}, \quad f \in C_c^\infty(\R^n) \otimes \B,
\end{equation*}
provided that $\B$ is a UMD Banach space. Here $C>0$ does not depend on $f$.

We define the operator
$$ T(f)
    = t^\beta \partial_t^\beta  P_t^{\mathfrak{L}_V}(f), \quad f \in C_c^\infty(\R^n)\otimes \B.$$
Since $C_c^\infty(\R^n)\otimes \B$ is a dense subspace of
$L^p(\R^n;\B)$, $T$ can be extended from $C_c^\infty(\R^n)\otimes \B$ to $L^p(\R^n,\B)$
as a bounded operator $\widetilde T$ from $L^p(\R^n,\B)$ into $T_p^2(\R^n,\B)$.
The same argument developed in Lemma~\ref{Lem2.1} allows us to obtain that
$$ \widetilde{T}f
    = t^\beta \partial_t^\beta P_t^{\mathfrak{L}_V}(f), \quad f \in L^p(\R^n,\B),$$
and then,
$$ \left\|t^\beta \partial_t^\beta P_t^{\mathfrak{L}_V}(f)\right\|_{T_p^2(\R^n,\B)}
    \leq C \|f\|_{L^p(\R^n,\B)}, \;\; f \in L^p(\R^n,\B). $$

\subsection{}
We now prove that, there exists $C>0$ for which
\begin{equation*}\label{M5}
    \|f\|_{L^p(\R^n,\B)}
        \leq C\left\|t^\beta \partial_t^\beta P_t^{\mathfrak{L}_V}(f)\right\|_{T^2_p(\R^n,\B)}, \;\; f\in L^p(\R^n,\B).
\end{equation*}

According to \cite[Proposition 2.1, ($ii$)]{BFRTT} we have that, for every $f,g \in L^2(\R^n)$,
$$ \int_{\R^n}\int_0^\infty t^\beta \partial_t^\beta P_t^{\mathfrak{L}_V}(f)(x) t^\beta \partial_t^\beta P_t^{\mathfrak{L}_V}(g)(x)\frac{dtdx}{t}
    = \frac{\Gamma(2\beta)}{2^{2\beta}}\int_{\mathbb R^n} f(x)g(x)\,dx.$$
Then, for every $f,g \in L^2(\R^n)$,
\begin{align*}
    \int_{\R^n}\int_{\Gamma(x)} t^\beta \partial_t^\beta P_t^{\mathfrak{L}_V}(f)(y) t^\beta \partial_t^\beta P_t^{\mathfrak{L}_V}(g)(y)\frac{dydt}{t^{n+1}}
        & = v_n \int_{\R^n}\int_0^\infty t^\beta \partial_t^\beta P_t^{\mathfrak{L}_V}(f)(y) t^\beta \partial_t^\beta P_t^{\mathfrak{L}_V}(g)(y)\frac{dydt}{t} \\
        & = v_n \frac{\Gamma(2\beta)}{2^{2\beta}} \int_{\R^n} f(y)g(y) dy.
\end{align*}
Hence, for every $f \in L^2(\R^n)\otimes \B$ and $g \in L^2(\R^n) \otimes \B^*$, we get
$$ \int_{\R^n} \int_{\Gamma(x)} \langle t^\beta \partial_t^\beta P_t^{\mathfrak{L}_V}(g)(y) , t^\beta \partial_t^\beta P_t^{\mathfrak{L}_V}(f)(y) \rangle_{\B^*,\B} \frac{dydt}{t^{n+1}}
    = v_n \frac{\Gamma(2\beta)}{2^{2\beta}} \int_{\R^n} \langle g(y)f(y)\rangle_{\B^*,\B} dy.$$
Now (\ref{M5}) can be established by proceeding as in the proof of Lemma~\ref{Lem2.2}.

\section{Proof of Theorem \ref{Theorem 1.2}, $(ii)$}\label{sec:Hermite}

For every $k \in \mathbb N$ we consider the $k$-th Hermite function defined by
$$ h_k(x)
    = (\sqrt{\pi} 2^kk!)^{-1/2} e^{-x^2/2} H_k(x), \quad x \in \mathbb R,$$
where $H_k$ represents the $k$-th Hermite polynomial (\cite[pp. 105--106]{Sze}).
If $k=(k_1,\ldots,k_n) \in \mathbb N^n$, the $k$-th Hermite function $h_k$ in $\R^n$ is defined by
$$ h_k(x)
    = \prod_{j=1}^n h_{k_j}(x_j), \quad x=(x_1,\ldots,x_n) \in \R^n.$$
For every $k \in \mathbb N^n$, $h_k$ is an eigenfunction of the Hermite operator $\mathcal{H}$ satisfying
$$ \mathcal{H}h_k
    =(2|k|+n) h_k,$$
where $|k|=k_1+\ldots+k_n$. The system $\{h_k\}_{k\in \mathbb N^n}$ is an orthonormal basis in $L^2(\R^n)$.

The heat semigroup associated with $\{h_k\}_{k\in \mathbb N^n}$ is defined by
$$ W_t^{\mathcal{H}}(f)(x)
    =\sum_{k \in \mathbb N^n} e^{-t(2|k|+n)} c_k(f)h_k(x), \quad f \in L^2(\R^n),$$
where
$$c_k(f)
    = \int_{\R^n} h_k(y) f(y) dy, \quad k \in \N^n.$$
According to the Mehler's formula (\cite[(1.1.36)]{Th}) we can write, for every $t>0$,
\begin{equation}\label{4.1}
    W_t^{\mathcal{H}}(f)(x)
        = \int_{\R^n} W_t^{\mathcal{H}}(x,y)f(y)dy, \quad f \in L^2(\R^n),
\end{equation}
where, for every $x,y \in \R^n$ and $t>0$,
$$ W_t^{\mathcal{H}}(x,y)
    =\frac{1}{\pi^{n/2}}\left(\frac{e^{-2t}}{1-e^{-4t}}\right)^{n/2} \exp\left[-\frac{1}{4}\left( |x-y|^2 \frac{1+e^{-2t}}{1-e^{-2t}}  + |x+y|^2\frac{1-e^{-2t}}{1+e^{-2t}}\right)\right].$$

We define, for each $t>0$, and $1 \leq p \leq \infty$, the operator $W_t^{\mathcal{H}}$ on $L^p(\R^n)$ by (\ref{4.1}).
Then, $\{W_t^{\mathcal{H}}\}_{t>0}$ is a positive semigroup of contractions in $L^p(\R^n)$, for every $1<p<\infty$.
$W_t^{\mathcal{H}}$ can be extended to $L^p(\R^n,\B)$ with the same boundedness properties, for every $t>0$ and $1<p<\infty$.

In the Hermite setting the critical radius $\rho(x)$, $x\in\R^n$, satisfies that
$$ \displaystyle\rho(x) \sim\left\{ \begin{array}{ll}
\displaystyle \frac{1}{1+|x|},& |x|\geq 1 \\
&\\
\displaystyle \frac{1}{2}, & |x| <1.\end{array}\right.$$

We are going to see that in this context we can establish properties that allow us to
prove Theorem \ref{Theorem 1.2}, $(ii)$, by proceeding as in the proof of Theorem \ref{Theorem 1.2}, $(i)$.
Note that now $n$ can be any nonnegative integer.

Firstly, according to Feynman-Kac formula we have that
\begin{equation}\label{4.2}
    |W_t^{\mathcal{H}}(x,y)|
        \leq C\frac{e^{-|x-y|^2/4t}}{t^{n/2}}, \;\; x,y \in \R^n \;\;\mbox{and}\;\, t>0.
\end{equation}

On the other hand, we have that
\begin{itemize}
    \item If $x,y \in \R^n$, $x\cdot y>0$, then $|x+y| \geq |y|$ and
    \begin{align*}
        |W_t^{\mathcal{H}}(x,y)|
            &\leq C\left(\frac{e^{-2t}}{1-e^{-4t}}\right)^{n/2}
                \exp\left[-\frac{1}{4}\left( |x-y|^2 \frac{1+e^{-2t}}{1-e^{-2t}}+|y|^2\frac{1-e^{-2t}}{1+e^{-2t}}\right)\right]\\
            &\leq C \frac{e^{-c|x-y|^2/t}}{t^{n/2}} \left(\frac{1+e^{-2t}}{1-e^{-2t}}\right)^{1/2}\frac{1}{|y|}
            \leq C \frac{e^{-c|x-y|^2/t}}{t^{n/2}}\frac{1}{\sqrt{t}|y|}\\
            &\leq C \frac{e^{-c|x-y|^2/t}}{t^{n/2}}\frac{\rho(y)}{\sqrt{t}}, \quad|y|>1\;\;\mbox{and}\;\;t>0.
    \end{align*}
    \item If $x,y \in \R^n$, $x\cdot y<0$, then $|x-y| \geq |y|$ and
    \begin{align*}
        |W_t^{\mathcal{H}}(x,y)|
            &\leq C\left(\frac{e^{-2t}}{1-e^{-4t}}\right)^{n/2}
                \exp\left(-\frac{1}{8}\left( |x-y|^2 \frac{1+e^{-2t}}{1-e^{-2t}} + |y|^2\frac{1+e^{-2t}}{1-e^{-2t}}\right)\right)\\
            &\leq C \frac{e^{-ct}}{t^{n/2}} \exp\left(-\frac{1}{8}|x-y|^2 \frac{1+e^{-2t}}{1-e^{-2t}}\right)
                \left(\frac{1-e^{-2t}}{1+e^{-2t}}\right)^{1/2}\frac{1}{|y|} \\
            &\leq C \frac{\sqrt{t}e^{-ct}}{|y|} \frac{e^{-c|x-y|^2/t}}{t^{n/2}}
            \leq C \frac{e^{-c|x-y|^2/t}}{t^{n/2}}\frac{\rho(y)}{\sqrt{t}}, \quad|y|>1\;\;\mbox{and}\;\;t>0.
    \end{align*}
    \item If $x,y \in \R^n$ and  $|y|<1$, then
    \begin{align*}
        |W_t^{\mathcal{H}}(x,y)|
            &\leq C \frac{e^{-c|x-y|^2/t}}{t^{n/2}}\frac{1}{\sqrt{t}}
             \leq C \frac{e^{-c|x-y|^2/t}}{t^{n/2}}\frac{\rho(y)}{\sqrt{t}}, \quad|y|\leq 1\;\;\mbox{and}\;\;t>0.
    \end{align*}
\end{itemize}
We conclude that
\begin{equation}\label{4.3}
    \left|W_t^{\mathcal{H}}(x,y)\right|
        \leq C\frac{e^{-c|x-y|^2/t}}{t^{n/2}}\frac{\rho(y)}{\sqrt{t}}, \quad x,y\in \R^n\;\; \mbox{and}\;\;t>0.
\end{equation}
The Poisson semigroup $\{P_t^{\mathcal{H}}\}_{t>0}$ associated with the Hermite
operator is defined, as usual, by subordination
$$P_t^{\mathcal{H}}(f)(x)
    = \frac{t}{\sqrt{4\pi}}\int_0^\infty \frac{e^{-t^2/4u}}{u^{3/2}}W_u^\mathcal{H}(f)(x)du, \quad f \in L^p(\R^n)\;\;\mbox{and}\;\;t>0.$$
The Poisson kernel $P_t^{\mathcal{H}}(x,y)$ can be written as
$$ P_t^{\mathcal{H}}(x,y)
    = \frac{t}{\sqrt{4\pi}}\int_0^\infty \frac{e^{-t^2/4u}}{u^{3/2}}W_u^\mathcal{H}(x,y)du, \quad x,y\in \R^n \;\;\mbox{and}\;\;t>0.$$
By using \eqref{3.3} and \eqref{4.3} we obtain
\begin{equation}\label{4.4}
    \left|t^\beta\partial_t^\beta P_t^{\mathcal{H}}(z,y)\right|
        \leq C\frac{\rho(y)t^\beta}{(t+|y-z|)^{\beta+n+1}}, \quad z,y \in \R^n\;\;\mbox{and}\;\; t>0.
\end{equation}
We also have that
\begin{align}\label{4.5}
    \int_{\R^n}|z|^2\frac{e^{-c|y-z|^2/s}}{s^{n/2}}dz
        &= \int_{\R^n}|y-w\sqrt{s}|^2 e^{-c|w|^2}dw
        = \int_{\R^n}(|y|^2+|w|^2s+2|y||w|\sqrt{s})e^{-c|w|^2}dw\nonumber\\
        &\leq \frac{C}{s}\frac{s}{\rho(y)^2}\int_{\R^n}(1+|w|^2+|w|)e^{-c|w|^2}dw
        \leq \frac{C}{s}\frac{s}{\rho(y)^2}, \quad 0<s<\rho(y)^2.
\end{align}
Note that $\rho(y) \leq 1$, $y\in \R^n$.

By proceeding as in the proof of \cite[Lemma 1.4]{Sh} we can see that there exists $1/2 \leq \gamma < 1$ such that
\begin{equation}\label{4.6}
    \rho(y)
        \leq C\rho(x) \left(1+\frac{|x-y|}{\rho(x)}\right)^\gamma, \quad x,y\in \R^n.
\end{equation}

Also, we can find a sequence $(x_k)_{k\in \mathbb N}$ in $\mathbb R^d$ such that if $Q_k=B(x_k,\rho(x_k))$ and
$Q^{**}_k=B(x_k,4\rho(x_k))$, $k \in \mathbb N$, the following properties hold (\cite[Lemma 2.3]{DZ1})
\begin{enumerate}
    \item $\cup_{k\in\mathbb N} Q_k=\R^n$,
    \item There exists $N\in\mathbb N$ such that, for every $k\in \mathbb N$, card$\{j\in \mathbb N: Q_j^{**}\cap Q_k^{**}\not=\emptyset\}\leq N$.
\end{enumerate}

These properties of the sequence of balls $\{Q_k\}_{k\in \mathbb N}$ and the estimates (\ref{4.2})-(\ref{4.6})
allow us to show Theorem \ref{Theorem 1.2} $(ii)$, by proceeding as in the proof of Theorem \ref{Theorem 1.2}, $(i)$.

\section{Proof of Theorem \ref{teo 1.3}}\label{sec:Bessel}

\subsection{}
Our first objective is to show that
\begin{equation}\label{5.1}
    \|t^\beta \partial _t^\beta P_t^{\mathfrak{B}_\lambda }(f)\|_{T_p^2((0,\infty ),\B)}
        \leq C\|f\|_{L^p((0,\infty ), \B)}, \quad f\in L^p((0,\infty ),\B).
\end{equation}
We can write
$$ B_\lambda
    =-\frac{d^2}{dx^2}+V(x),\quad x\in (0,\infty ),$$
where $V(x)=\lambda (\lambda -1)/x^2$.
Then, $\partial _tP_t^{\mathfrak{B}_\lambda}(1)\not=0$, and in order to prove (\ref{5.1})
we cannot use \cite[Theorem 4.8]{HNP}. We are going to proceed as in Section~\ref{sec:Schrodinger},
by comparing the operator $t^\beta \partial _t^\beta P_t^{\mathfrak{B}_\lambda }(f)$
with the one related to the one-dimensional classical Poisson semigroup given by \eqref{eq:Poissinteg}.
In the following lemmas we collect some estimates that will be very helpful
for our purposes.

\begin{lema}\label{Lem5.1}
    Let $\beta>0$. Then,
    $$\Big \| t^\beta \partial_t^\beta P_t(y \pm z) \chi_{\Gamma_+(x)}(y,t)\Big\|_{H_+}
        \leq \frac{C}{|x \pm z|}, \quad x,z \in (0,\infty),$$
    where $\Gamma_+(x)=\{(y,t) \in (0,\infty)^2 : |x-y|<t\}$.
\end{lema}

\begin{proof}
    In \cite[Lemma 2]{BCCFR1} it was established that
    \begin{equation}\label{5.2}
        t^\beta \partial _t^\beta P_t(z)
            =\sum_{k=0}^{(m+1)/2}\frac{c_k}{t}\varphi ^k\Big(\frac{z}{t}\Big),\quad z\in \mathbb{R} \mbox{ and } t>0,
    \end{equation}
    where $m\in \mathbb{N}$ is such that $m-1\leq \beta <m$, and, for every $k\in \mathbb{N}$,
    $0\leq k\leq (m+1)/2$,
    $c_k\in \mathbb{C}$ and
    $$ \varphi ^k(z)
        =\int_0^\infty \frac{(1+v)^{m+1-2k}v^{m-\beta -1}}{((1+v)^2+z^2)^{m-k+1}}dv,\quad z\in\mathbb{R}.$$
    Let $k\in \mathbb{N}$, $0\leq k\leq (m+1)/2$. We can write
    $$\frac{1}{t}\varphi ^k\Big(\frac{y+z}{t}\Big)
        =t^{2(m-k)+1}\int_0^\infty \frac{(1+v)^{m+1-2k}v^{m-\beta -1}}{((1+v)^2t^2+(y+z)^2)^{m-k+1}}dv,\quad t,y,z\in (0,\infty ).$$
    By using Minkowski's inequality we obtain
    \begin{align*}
        &\left(\int_{\Gamma _+(x)}\left|\frac{1}{t}\varphi ^k\Big(\frac{y+z}{t}\Big)\right|^2\frac{dydt}{t^2}\right)^{1/2} \\
        & \qquad \leq C\int_0^\infty (1+v)^{m+1-2k}v^{m-\beta -1}\left(\int_0^\infty \int_{|x-y|<t}\frac{t^{4(m-k)}}{(t+ y+z +tv)^{4(m-k+1)}}dydt\right)^{1/2}dv\\
        & \qquad \leq C\int_0^\infty (1+v)^{m+1-2k}v^{m-\beta -1}\left(\int_0^\infty \int_{|x-y|<t}\frac{t^{4(m-k)}}{(t+|x-y|+y+z+tv)^{4(m-k+1)}}dydt\right)^{1/2}dv\\
        & \qquad \leq C\int_0^\infty (1+v)^{m+1-2k}v^{m-\beta -1}\left(\int_0^\infty \frac{t^{4(m-k)+1}}{(t(1+v)+x+z)^{4(m-k+1)}}dt\right)^{1/2}dv\\
        & \qquad \leq \frac{C}{x+z}\int_0^\infty \frac{v^{m-\beta -1}}{(1+v)^m}dv,\quad x,z\in (0,\infty ).
    \end{align*}
    Hence,
    \begin{equation*}\label{5.4}
      \Big \| t^\beta \partial_t^\beta P_t(y+z) \chi_{\Gamma_+(x)}(y,t)\Big\|_{H_+}
            \leq \frac{C}{x+z},\quad x,z\in (0,\infty ).
    \end{equation*}
    By taking into account that $|x-y|+|y-z|\geq |x-z|$, $x,y,z\in (0,\infty )$, the above arguments allow us to obtain that
    \begin{equation*}\label{5.5}
       \Big \| t^\beta \partial_t^\beta P_t(y-z) \chi_{\Gamma_+(x)}(y,t)\Big\|_{H_+}\leq \frac{C}{|x-z|}.
    \end{equation*}
\end{proof}

It is common to decompose the Bessel-Poisson kernel as follows
\begin{align*}
    P_t^{\mathfrak{B}_\lambda}(y,z)
        & = \frac{2\lambda (yz)^\lambda t}{\pi} \bigg( \int_0^{\pi /2} + \int_{\pi /2}^\pi \bigg)
                \frac{(\sin \theta )^{2\lambda -1}}{(t^2 + (y-z)^2+2yz(1-\cos \theta ))^{\lambda +1}}d\theta \\
        & = P_t^{\mathfrak{B}_\lambda, 1}(y,z) + P_t^{\mathfrak{B}_\lambda, 2}(y,z), \quad t,y,z\in (0,\infty ).
\end{align*}

\begin{lema}\label{Lem5.2}
    Let $\beta, \lambda >0$. Then,
    $$\Big \| t^\beta \partial_t^\beta P_t^{\mathfrak{B}_\lambda, 1}(y,z) \chi_{\Gamma_+(x)}(y,t)\Big\|_{H_+}
        \leq \frac{C}{x- z}, \quad x,z \in (0,\infty),$$
    and
    $$\Big \| t^\beta \partial_t^\beta P_t^{\mathfrak{B}_\lambda, 2}(y,z) \chi_{\Gamma_+(x)}(y,t)\Big\|_{H_+}
        \leq \frac{C}{|x+z|}, \quad x,z \in (0,\infty).$$
\end{lema}

\begin{proof}
    By \cite[(27)]{BCR3} we have that, for each $t,x,y \in (0,\infty)$,
    \begin{equation}\label{5.3}
        t^\beta \partial _t^\beta P_t^{\mathfrak{B}_\lambda }(x,y)
            =\sum_{k=0}^{(m+1)/2}\frac{b_k^\lambda }{t^{2\lambda +1}}(xy)^\lambda \int_0^\pi (\sin \theta )^{2\lambda -1}\varphi ^{\lambda , k}\left(\frac{\sqrt{(x-y)^2+2xy(1-\cos \theta )}}{t}\right)d\theta,
    \end{equation}
    where $m \in \N$ is such that $m-1 \leq \beta < m$ and, for every $k\in \mathbb{N}$, $0\leq k\leq (m+1)/2$,
    $$ \varphi ^{\lambda ,k}(z)
        =\int_0^\infty \frac{(1+v)^{m+1-2k}v^{m-\beta -1}}{((1+v)^2+z^2)^{\lambda +m-k+1}}dv,\quad z\in (0,\infty),$$
    and
    $$b_k^\lambda
        =\frac{2\lambda (\lambda +1)\cdots (\lambda +m-k)}{(m-k)!}c_k.$$
    Here $c_k$ is as in (\ref{5.2}).
    By (\ref{5.3}) we get
    $$ t^\beta \partial_t^\beta P_t^{\mathfrak{B}_\lambda, 2}(y,z)
        =\sum_{k=0}^{(m+1)/2}\frac{b_k^\lambda }{t^{2\lambda +1}}(yz)^\lambda \int_{\pi /2}^\pi (\sin \theta )^{2\lambda -1}
            \varphi^{\lambda ,k}\left(\frac{\sqrt{(y-z)^2+2yz(1-\cos \theta )}}{t}\right)d\theta,$$
    for each $t,y,z\in (0,\infty )$. Let $k\in \mathbb{N}$, $0\leq k\leq (m+1)/2$.
    We have that, for every $t,y,z\in (0,\infty )$,
    \begin{align*}
        &\left|\frac{(yz)^\lambda }{t^{2\lambda +1}}\int_{\pi /2}^\pi (\sin \theta )^{2\lambda -1}\int_0^\infty \frac{(1+v)^{m+1-2k}v^{m-\beta -1}}{((1+v)^2+\frac{(y-z)^2+2yz(1-\cos \theta ))}{t^2})^{\lambda +m-k+1}}dvd\theta \right|\\
        & \qquad \leq C(yz)^\lambda t^{2m-2k+1}\int_{\pi /2}^\pi (\sin \theta )^{2\lambda -1}\int_0^\infty \frac{(1+v)^{m+1-2k}v^{m-\beta -1}}{(t^2(1+v)^2+(y-z)^2+2yz)^{\lambda +m-k+1}}dvd\theta .
    \end{align*}
    Hence, Minkowski's inequality leads to
    \begin{align*}
        &\left(\int_{\Gamma _+(x)}\left|\frac{(yz)^\lambda }{t^{2\lambda +1}}\int_{\pi /2}^\pi (\sin \theta )^{2\lambda -1}
            \varphi^{\lambda ,k}\left(\frac{\sqrt{(y-z)^2+2yz(1-\cos \theta )}}{t}\right)d\theta\right|^2\frac{dydt}{t^2}\right)^{1/2}\\
        & \qquad \leq C\left(\int_{\Gamma _+(x)}\left((yz)^\lambda t^{2m-2k+1}\int_0^\infty \frac{(1+v)^{m+1-2k}v^{m-\beta -1}}{(t^2(1+v)^2+y^2+z^2)^{\lambda +m-k+1}}dv\right)^2\frac{dydt}{t^2}\right)^{1/2}\\
        & \qquad \leq C\int_0^\infty (1+v)^{m+1-2k}v^{m-\beta -1}\left(\int_{\Gamma _+(x)}\left(\frac{t^{2m-2k+1}}{(t^2(1+v)^2+y^2+z^2)^{m-k+1}}\right)^2\frac{dydt}{t^2}\right)^{1/2}dv\\
        & \qquad \leq C\int_0^\infty (1+v)^{m+1-2k}v^{m-\beta -1}\left(\int_0^\infty \frac{t^{4(m-k)+1}}{(t(1+v)+x+z)^{4(m-k+1)}}dt\right)^{1/2}dv\\
        & \qquad \leq \frac{C}{x+z},\quad x,z\in (0,\infty ).
    \end{align*}
    Then,
    \begin{equation*}\label{5.6}
        \left(\int_{\Gamma _+(x)}|t^\beta \partial _t^\beta P_t^{\mathfrak{B}_{\lambda ,2}}(y,z)|^2\frac{dydt}{t^2}\right)^{1/2}
            \leq \frac{C}{x+z},\quad x,z\in (0,\infty ).
    \end{equation*}
    In a similar way we can see that
    \begin{equation*}\label{5.7}
        \left(\int_{\Gamma _+(x)}|t^\beta \partial _t^\beta P_t^{\mathfrak{B}_\lambda, 1}(y,z)|^2\frac{dydt}{t^2}\right)^{1/2}
        \leq \frac{C}{|x-z|}, \quad x,z\in (0,\infty ).
    \end{equation*}
\end{proof}

\begin{lema}\label{Lem5.3}
    Let $\beta, \lambda >0$. Then,
    $$\Big \| t^\beta \partial_t^\beta [P_t^{\mathfrak{B}_\lambda, 1}(y,z) - P_t(y-z)]
            \chi_{\Gamma_+(x)}(y,t)\Big\|_{H_+}
        \leq \frac{C}{z} \Big( 1 + \log_+ \frac{z}{|x-z|} \Big), \quad 0< \frac{x}{2} < z < 2x.$$
\end{lema}

\begin{proof}
    We use the following decomposition
    \begin{align*}
        & P_t^{\mathfrak{B}_\lambda, 1}(y,z)
            = \frac{2\lambda (yz)^\lambda t}{\pi }\int_0^{\pi /2}\frac{(\sin \theta )^{2\lambda -1}-\theta ^{2\lambda -1}}{(t^2 + (y-z)^2+2yz(1-\cos \theta ))^{\lambda +1}}d\theta \\
        & \quad  +\frac{2\lambda (yz)^\lambda t}{\pi }\int_0^{\pi /2}\theta ^{2\lambda -1}\left(\frac{1}{(t^2 + (y-z)^2+2yz(1-\cos \theta ))^{\lambda +1}}\right.
            -\left.\frac{1}{(t^2 + (y-z)^2+yz\theta ^2)^{\lambda +1}}\right)d\theta\\
        & \quad + \frac{2\lambda (yz)^\lambda t}{\pi }\int_0^{\pi /2}\frac{\theta ^{2\lambda -1}}{(t^2 + (y-z)^2+yz\theta ^2)^{\lambda +1}}d\theta ,\quad t,y,z\in (0,\infty ).
    \end{align*}
    On the other hand, we get (see \cite[p. 485]{BCFR1}) for every $t,y,z\in (0,\infty )$,
    \begin{align*}
        \int_0^{\pi /2}\frac{\theta ^{2\lambda -1}}{(t^2 + (y-z)^2+yz\theta ^2)^{\lambda +1}}d\theta
            &= \left(\int_0^\infty -\int_{\pi /2}^\infty \right)\frac{\theta ^{2\lambda -1}}{(t^2 + (y-z)^2+yz\theta ^2)^{\lambda +1}}d\theta\\
        &  = \frac{\pi}{2\lambda (yz)^\lambda t} P_t(y-z)
                    -\int_{\pi /2}^\infty \frac{\theta ^{2\lambda -1}}{(t^2 + (y-z)^2+yz\theta ^2)^{\lambda +1}}d\theta.
    \end{align*}
    Hence,
    $$t^\beta \partial_t^\beta [P_t^{\mathfrak{B}_\lambda, 1}(y,z)-P_t(y-z)]
        =\sum_{j=1}^3S_j(y,z,t),\quad t,y,z\in (0,\infty ),$$
    where
    $$S_1(y,z,t)
        =t^\beta \partial _t^\beta \left(\frac{2\lambda (yz)^\lambda t}{\pi }\int_0^{\pi /2}\frac{(\sin \theta )^{2\lambda -1}-\theta ^{2\lambda -1}}{(t^2 + (y-z)^2+2yz(1-\cos \theta ))^{\lambda +1}}d\theta \right),$$
    \begin{align*}
        S_2(y,z,t)
            & =t^\beta \partial _t^\beta \left(\frac{2\lambda (yz)^\lambda t}{\pi }\int_0^{\pi /2}\frac{\theta ^{2\lambda -1}}{(t^2 + (y-z)^2+2yz(1-\cos \theta ))^{\lambda +1}}\right.\\
            & \qquad \qquad \left.-\frac{\theta ^{2\lambda -1}}{(t^2 + (y-z)^2+yz\theta ^2)^{\lambda +1}}d\theta\right),
    \end{align*}
    and
    $$ S_3(y,z,t)
        = - t^\beta \partial _t^\beta \left(\frac{2\lambda (yz)^\lambda t}{\pi }\int_{\pi /2}^\infty \frac{\theta ^{2\lambda -1}}{(t^2 + (y-z)^2+yz\theta ^2)^{\lambda +1}}d\theta\right).$$

    Assume that $0<x/2<z<2x$. We are going to analyze $S_1$, $S_2$ and $S_3$ separately.

    From (\ref{5.3}) we deduce that, for every $t \in (0,\infty)$,
    $$S_1(y,z,t)
        =\sum_{k=0}^{(m+1)/2}\frac{b_k^\lambda }{t^{2\lambda +1}}(yz)^\lambda \int_0^{\pi/2} [(\sin \theta )^{2\lambda -1}-\theta ^{2\lambda -1}]
            \varphi ^{\lambda ,k}\left(\frac{\sqrt{(y-z)^2+2yz(1-\cos \theta )}}{t}\right)d\theta,$$
    where $m \in \N$ is such that $m-1 \leq \beta < m$.
    Let $k\in \mathbb{N}$, $0\leq k\leq (m+1)/2$. Since
    $$|(\sin \theta)^{2\lambda -1}-\theta ^{2\lambda -1}|\leq C\theta ^{2\lambda +1}
        \quad \text{and} \quad
        |1-\cos \theta |\geq C\theta ^2, \qquad\theta \in [0,\pi /2],$$
    from \cite[p. 60--61]{MS} we obtain
    \begin{align*}
        &\left|\int_0^{\pi /2}\frac{(\sin \theta )^{2\lambda -1}-\theta ^{2\lambda -1}}{(t^2(1+v)^2+(y-z)^2+2yz(1-\cos \theta ))^{\lambda+m-k+1}}d\theta \right|\\
        &\qquad\leq C\int_0^{\pi /2}\frac{\theta ^{2\lambda +1}}{(t^2(1+v)^2+(y-z)^2+yz\theta ^2)^{\lambda+m-k+1}}d\theta \\
        &\qquad\leq  C\frac{(yz)^{-\lambda -1}}{(t^2(1+v)^2+(y-z)^2)^{m-k}}\int_0^{\frac{\pi}{2}\sqrt{\frac{yz}{t^2(1+v)^2+(y-z)^2}}}\frac{u^{2\lambda +1}}{(1+u^2)^{\lambda+m-k+1}}du\\
        &\qquad\leq C(yz)^{-\lambda -1}\left(\frac{\sqrt{\frac{yz}{t^2(1+v)^2+(y-z)^2}}}{1+ \sqrt{\frac{yz}{t^2(1+v)^2+(y-z)^2}}}\right)^{2\lambda +2}
        \left\{\begin{array}{ll}
        \displaystyle \frac{1}{(t^2(1+v)^2+(y-z)^2)^{m-k}},&m>k,\\
                &\\
        \displaystyle 1+\log _+\frac{yz}{t^2(1+v)^2+(y-z)^2},&m=k,
        \end{array}
        \right.\\
        &\qquad\leq \frac{C}{(t^2(1+v)^2+(y-z)^2+yz)^{\lambda +1}}
        \left\{\begin{array}{ll}
        \displaystyle \frac{1}{(t^2(1+v)^2+(y-z)^2)^{m-k}},&m>k,\\
                &\\
        \displaystyle 1+\log _+\frac{yz}{t^2(1+v)^2+(y-z)^2},&m=k,
        \end{array}    \right.
    \end{align*}
    for every $t,z,y\in (0,\infty )$. Notice that, since $0\leq k\leq (m+1)/2$; $k=m$ if, and only if, $k=m=1$.

    Suppose that $m>k$. We have that
    \begin{align*}
        & \left(\int_{\Gamma _+(x)}\left(\frac{(yz)^\lambda }{t^{2\lambda +1}}\int_0^{\pi /2}[(\sin \theta )^{2\lambda -1}-\theta ^{2\lambda -1}]
            \varphi^{\lambda,k}\left(\frac{\sqrt{(y-z)^2+2yz(1-\cos \theta)}}{t}\right)d\theta \right)^2\frac{dydt}{t^2}\right)^{1/2}\\
        & \qquad \leq C\int_0^\infty (1+v)^{m+1-2k}v^{m-\beta -1}\\
        &\qquad \qquad \times \left(\int_{\Gamma _+(x)}\left(\frac{(yz)^\lambda t^{2(m-k)+1}}{[t^2(1+v)^2 + (y-z)^2 +  yz]^{\lambda +1}[t^2(1+v)^2+(y-z)^2]^{m-k}}\right)^2\frac{dydt}{t^2}\right)^{1/2}dv\\
        & \qquad \leq C\int_0^\infty (1+v)^{m+1-2k}v^{m-\beta -1} \bigg[\left(\int_0^{z/2}+\int_{z/2}^{2z}+\int_{2z}^\infty \right) \\
        & \qquad \qquad \times \int_{|x-y|}^\infty \frac{(yz)^{2\lambda} t^{4(m-k)}}{[t^2(1+v)^2 + (y-z)^2 +  yz]^{2\lambda +2}[t^2(1+v)^2+(y-z)^2]^{2(m-k)}}dtdy \bigg]^{1/2}dv\\
        & \qquad = \sum_{j=1}^3I_j(x,z).
    \end{align*}
    We can write
    \begin{align*}
        &I_1(x,z)+I_3(x,z)\\
        &\leq C\int_0^\infty\frac{(1+v)^{m+1-2k}v^{m-\beta -1}}{(1+v)^{2(m-k)}}\left[\left(\int_0^{z/2}+\int_{2z}^\infty \right)\int_{|x-y|}^\infty
        \frac{(yz)^{2\lambda }}{(t^2(1+v)^2 + (y-z)^2 +  yz)^{2\lambda +2}}dtdy\right]^{1/2}dv\\
        &\leq C\int_0^\infty\frac{v^{m-\beta -1}}{(1+v)^{m-1}}\left[\left(\int_0^{z/2}+\int_{2z}^\infty \right)\int_{|x-y|}^\infty
        \frac{dtdy}{(|y-z|+t(1+v))^{4}}\right]^{1/2}dv\\
        &\leq C\int_0^\infty\frac{v^{m-\beta -1}}{(1+v)^{m-1/2}}\left[\left(\int_0^{z/2}+\int_{2z}^\infty \right)\int_{|x-y|(1+v)}^\infty \frac{dwdy}{(|y-z|+w)^{4}}\right]^{1/2}dv\\
        &\leq C\int_0^\infty\frac{v^{m-\beta -1}}{(1+v)^{m-1/2}}\left[\left(\int_0^{z/2}+\int_{2z}^\infty \right)\frac{dy}{(|y-z|+|x-y|(1+v))^{3}}\right]^{1/2}dv\\
        &\leq C\int_0^\infty\frac{v^{m-\beta -1}}{(1+v)^{m-1/2}}\left[\left(\int_0^{z/2}+\int_{2z}^\infty \right)\frac{dy}{(z+|x-y|(1+v))^{3}}\right]^{1/2}dv\\
        &\leq C\int_0^\infty\frac{v^{m-\beta -1}}{(1+v)^m}dv\left[\int_0^\infty \frac{du}{(z+u)^{3}}\right]^{1/2}\leq \frac{C}{z}.
    \end{align*}
    Also we get
    \begin{align*}
        I_2(x,z)
            &\leq C\int_0^\infty \frac{v^{m-\beta -1}}{(1+v)^{m-1}}\left[\int_{z/2}^{2z}\int_{|x-y|}^\infty \frac{(yz)^{2\lambda }}{(t^2(1+v)^2+|y-z|^2+yz)^{2\lambda +2}}dtdy\right]^{1/2}dv\\
            &\leq C\int_0^\infty \frac{v^{m-\beta -1}}{(1+v)^{m-1}}\left[\int_{z/2}^{2z}\int_{|x-y|}^\infty \frac{dtdy}{(\sqrt{yz}+t(1+v))^4}\right]^{1/2}dv\\
            &\leq C\int_0^\infty \frac{v^{m-\beta -1}}{(1+v)^{m-1/2}}\left[\int_{z/2}^{2z} \frac{dy}{(z+|x-y|(1+v))^3}\right]^{1/2}dv
                \leq \frac{C}{z}.
    \end{align*}
    Hence,
    $$\left(\int_{\Gamma _+(x)}\left(\frac{(yz)^\lambda }{t^{2\lambda +1}}\int_0^{\pi /2}[(\sin \theta )^{2\lambda -1}-\theta ^{2\lambda -1}]
        \varphi^{\lambda,k}\left(\frac{\sqrt{(y-z)^2+2yz(1-\cos \theta)}}{t}\right)d\theta \right)^2\frac{dydt}{t^2}\right)^{1/2}
        \leq \frac{C}{z}.$$

    Assume now $k=m=1$. Then, $0<\beta <1$. We have that
    \begin{align*}
        & \left(\int_{\Gamma _+(x)}\left(\frac{(yz)^\lambda }{t^{2\lambda +1}}\int_0^{\pi /2}[(\sin \theta )^{2\lambda -1}-\theta ^{2\lambda -1}]
            \varphi^{\lambda,1}\left(\frac{\sqrt{(y-z)^2+2yz(1-\cos \theta)}}{t}\right)d\theta \right)^2\frac{dydt}{t^2}\right)^{1/2}&&\\
        & \qquad \leq C\int_0^\infty v^{-\beta}\left[\int_{\Gamma _+(x)}\left(\frac{(yz)^\lambda t}{(t^2(1+v)^2 + (y-z)^2 +  yz)^{\lambda +1}}\left(1+\log _+\frac{yz}{t^2(1+v)^2+(y-z)^2}\right)\right)^2\frac{dydt}{t^2}\right]^{1/2}dv\\
        &\qquad \leq C\int_0^\infty v^{-\beta}\left[\left(\int_0^{z/2}+\int_{z/2}^{2z}+\int_{2z}^\infty \right)\int_{|x-y|}^\infty \frac{(yz)^{2\lambda}}{(t^2(1+v)^2 + (y-z)^2 +  yz)^{2\lambda +2}}\right.\\
        &\qquad \qquad \times \left. \left(1+\log _+\frac{yz}{t^2(1+v)^2+(y-z)^2}\right)^2dtdy\right]^{1/2}dv\\
        & \qquad = \sum_{j=1}^3 J_j(x,z).
    \end{align*}
    We can write
    \begin{align*}
        & J_1(x,z)+J_3(x,z)
            \leq C\int_0^\infty v^{-\beta}\left[\left(\int_0^{z/2}+\int_{2z}^\infty \right)\int_{|x-y|}^\infty \frac{(yz)^{2\lambda}}{(t^2(1+v)^2 + (y-z)^2 +  yz)^{2\lambda +2}}dtdy\right]^{1/2}dv\\
        & \qquad \leq C\int_0^\infty v^{-\beta}\left[\left(\int_0^{z/2}+\int_{2z}^\infty \right)\int_{|x-y|}^\infty \frac{1}{(|y-z|+t(1+v))^4}dtdy\right]^{1/2}dv
                \leq \frac{C}{z}\int_0^\infty \frac{v^{-\beta }}{1+v}dv
                \leq \frac{C}{z},
    \end{align*}
    and
    \begin{align*}
        J_2(x,z)
            \leq & C\int_0^\infty v^{-\beta}\left[\int_{z/2}^{2z}\int_{|x-y|}^\infty \frac{(yz)^{2\lambda}}{(t^2(1+v)^2 + (y-z)^2 +  yz)^{2\lambda +2}}\right.\\
            &\times \left.\left(1+\log_+\frac{yz}{t^2(1+v)^2+(y-z)^2}\right)^2dtdy\right]^{1/2}dv\\
            \leq &C\int_0^\infty v^{-\beta}\left[\int_{z/2}^{2z}\int_{|x-y|}^\infty  \frac{(yz)^{2\lambda}}{(t^2(1+v)^2 + (y-z)^2 +  yz)^{2\lambda +2}}\right.\\
            &\times \left.\left(1+\log_+\frac{yz}{t^2(1+v)^2+(x-y)^2+(y-z)^2}\right)^2dtdy\right]^{1/2}dv\\
            \leq & C\left(1+\log_+\frac{z}{|x-z|}\right)\int_0^\infty v^{-\beta}\left[\int_{z/2}^{2z}\int_{|x-y|}^\infty  \frac{1}{(z+t(1+v))^4}dtdy\right]^{1/2}dv\\
            \leq &\frac{C}{z}\left(1+\log_+\frac{z}{|x-z|}\right).
    \end{align*}
    Hence
    \begin{align*}
        & \left(\int_{\Gamma _+(x)}\left(\frac{(yz)^\lambda }{t^{2\lambda +1}}\int_0^{\pi /2}[(\sin \theta )^{2\lambda -1}-\theta ^{2\lambda -1}]
            \varphi ^{\lambda ,1}\left(\frac{\sqrt{(y-z)^2+2yz(1-\cos \theta )}}{t}\right)d\theta \right)^2\frac{dydt}{t^2}\right)^{1/2}\\
        & \qquad \qquad \leq \frac{C}{z}\left(1+\log_+\frac{z}{|x-z|}\right).
    \end{align*}
    We conclude that
    \begin{equation}\label{5.8}
        \left(\int_{\Gamma _+(x)}|S_1(y,z,t)|^2\frac{dydt}{t^2}\right)^{1/2}
            \leq \frac{C}{z}\left(1+\log_+\frac{z}{|x-z|}\right).
    \end{equation}

    Next we treat $S_2$.
    Let $k\in \mathbb{N}$, $0\leq k\leq (m+1)/2$. By using the mean value theorem we obtain
    \begin{align*}
        &\left|\frac{1}{((1+v)^2+\frac{(y-z)^2+2yz(1-\cos \theta )}{t^2})^{\lambda +m-k+1}}-\frac{1}{((1+v)^2+\frac{(y-z)^2+yz\theta ^2}{t^2})
        ^{\lambda +m-k+1}}\right|\\
        &\qquad = t^{2(\lambda +m-k+1)}\left |\frac{1}{(t^2(1+v)^2+(y-z)^2+2yz(1-\cos \theta ))^{\lambda +m-k+1}}\right.\\
        &\qquad \qquad \left. -\frac{1}{(t^2(1+v)^2+(y-z)^2+yz\theta ^2)^{\lambda +m-k+1}}\right|\\
        &\qquad \leq Ct^{2(\lambda +m-k+1)}\frac{yz\theta^4}{(t^2(1+v)^2+(y-z)^2+yz\theta ^2)^{\lambda +m-k+2}},
        \quad \theta \in (0,\pi/2), \ t,y\in (0,\infty ),
    \end{align*}
    because $|1-\cos \theta -\theta ^2/2|\leq C\theta ^4$, $\theta \in (0,\pi /2)$.
    Hence
    \begin{align*}
        & \left|\int_0^{\pi /2}\theta ^{2\lambda -1}\left(\frac{1}{((1+v)^2+\frac{(y-z)^2+2yz(1-\cos \theta )}{t^2})^{\lambda +m-k+1}}-\frac{1}{((1+v)^2+\frac{(y-z)^2+yz\theta ^2}{t^2})^{\lambda +m-k+1}}\right)d\theta \right|\\
        & \qquad \leq Ct^{2(\lambda +m-k+1)}\int_0^{\pi /2}\frac{\theta ^{2\lambda +1}}{(t^2(1+v)^2+(y-z)^2+yz\theta^2)^{\lambda +m-k+1}}d\theta,
        \quad t,y\in (0,\infty ).
    \end{align*}
    By proceeding as in the previous case we get
    \begin{equation}\label{5.9}
        \left(\int_{\Gamma _+(x)}|S_2(y,z,t)|^2\frac{dydt}{t^2}\right)^{1/2}\leq \frac{C}{z}\left(1+\log_+\frac{z}{|x-z|}\right).
    \end{equation}

    Finally we consider $S_3$. From (\ref{5.3}) it follows that
    $$S_3(y,z,t)
        =-\sum_{k=0}^{(m+1)/2}\frac{b_k^\lambda }{t^{2\lambda +1}}(yz)^\lambda \int_{\pi /2}^\infty \theta ^{2\lambda -1}
            \varphi ^{\lambda ,k}\left(\frac{\sqrt{(y-z)^2+yz\theta ^2}}{t}\right)d\theta ,\quad y,t\in (0,\infty ),$$
    where $m \in \N$ is such that $m-1 \leq \beta < m$.
    Let $k\in \mathbb{N}$, $0\leq k\leq (m+1)/2$. We can write
    \begin{align*}
        & \int_{\pi /2}^\infty \theta ^{2\lambda -1}\varphi ^{\lambda ,k}\left(\frac{\sqrt{(y-z)^2+yz\theta ^2}}{t}\right)d\theta\\
        & \qquad \leq \int_0^\infty (1+v)^{m+1-2k}v^{m-\beta -1}\int_{\pi /2}^\infty \frac{\theta ^{2\lambda -1}t^{2(\lambda +m-k+1)}}{(t^2(1+v)^2+(y-z)^2+yz\theta ^2)^{\lambda +m-k+1}}d\theta dv,\quad t,y\in (0,\infty ).
    \end{align*}
    Then,
    \begin{align*}
        & \left(\int_{\Gamma _+(x)}\left(\frac{(yz)^\lambda }{t^{2\lambda +1}}\int_{\pi /2}^\infty \theta ^{2\lambda -1}
            \varphi ^{\lambda ,k}\left(\frac{\sqrt{(y-z)^2+yz\theta ^2}}{t}\right)d\theta \right)^2\frac{dydt}{t^2}\right)^{1/2}\\
        & \qquad \leq \int_0^\infty (1+v)^{m+1-2k}v^{m-\beta -1}\int_{\pi /2}^\infty
            \left[\int_{\Gamma _+(x)}\frac{(yz)^{2\lambda} \theta ^{4\lambda -2}t^{4(m-k)}}{(t^2(1+v)^2+(y-z)^2+yz\theta ^2)^{2(\lambda +m-k+1)}}dydt\right]^{1/2}d\theta dv\\
        & \qquad \leq \int_0^\infty \frac{v^{m-\beta -1}}{(1+v)^{m-1}}\int_{\pi /2}^\infty \frac{1}{\theta}
            \left[\int_0^\infty\int_{|x-y|}^\infty \frac{dtdy}{(t(1+v)+|y-z|+\sqrt{yz}\theta)^4}\right]^{1/2}d\theta dv\\
        & \qquad \leq \int_0^\infty \frac{v^{m-\beta -1}}{(1+v)^{m-1/2}}\int_{\pi /2}^\infty \frac{1}{\theta}
            \left[\int_0^\infty \frac{dy}{((1+v)|x-y|+|y-z|+\sqrt{yz}\theta)^3}\right]^{1/2}d\theta dv\\
        & \qquad \leq \int_0^\infty \frac{v^{m-\beta -1}}{(1+v)^{m-1/2}}\int_{\pi /2}^\infty \frac{1}{\theta}
            \left[\left(\int_0^{z/4}+\int_{z/4}^\infty \right)\frac{dy}{((1+v)|x-y|+|y-z|+\sqrt{yz}\theta)^3}\right]^{1/2}d\theta dv.
    \end{align*}
    Since $x/2<z<2x$, if $y<z/4$, then $y<x/2$ and we have that
    \begin{align*}
        &\int_0^{z/4}\frac{dy}{((1+v)|x-y|+|y-z|+\sqrt{yz}\theta)^3}
            \leq C\int_0^{z/4}\frac{dy}{((1+v)x+z+\sqrt{yz}\theta)^3}\\
        & \qquad \leq C\int_0^{z/4}\frac{dy}{((1+v)^2x^2+z^2+yz\theta^2)^{3/2}}
        \leq \frac{C}{\theta ^2z((1+v)x+z)},\quad \theta \in (\pi/2,\infty)\mbox{ and }v\in (0,\infty ).
    \end{align*}
    Also we get
    \begin{align*}
        & \int_{z/4}^\infty \frac{dy}{((1+v)|x-y|+|y-z|+\sqrt{yz}\theta)^3}
            \leq C\int_{z/4}^\infty \frac{1}{((1+v)|x-y|+z\theta)^3}dy\\
        &\qquad \qquad \leq C\int_0^\infty \frac{dw}{((1+v)w+z\theta)^3}
            \leq \frac{C}{(1+v)(z \theta)^2},\quad \theta \in (\pi/2,\infty)\mbox{ and }v\in (0,\infty ).
    \end{align*}
    Hence,
    \begin{align*}
        &\left(\int_{\Gamma _+(x)}\left(\frac{(yz)^\lambda }{t^{2\lambda +1}}\int_{\pi /2}^\infty \theta ^{2\lambda -1}
            \varphi ^{\lambda ,k}\left(\frac{\sqrt{(y-z)^2+yz\theta ^2}}{t}\right)d\theta \right)^2\frac{dydt}{t^2}\right)^{1/2}\\
        & \qquad \qquad \leq C\int_0^\infty \frac{v^{m-\beta -1}}{(1+v)^{m-1/2}}\int_{\pi /2}^\infty \frac{1}{\theta ^2}\left(\frac{1}{z((1+v)x+z)}+\frac{1}{(1+v)z^2}\right)^{1/2}d\theta dv\\
        &\qquad \qquad \leq C\left(\frac{1}{\sqrt{z}}\int_0^\infty \frac{v^{m-\beta -1}}{(1+v)^{m-1/2}((1+v)x+z)^{1/2}}dv+\frac{1}{z}\int_0^\infty \frac{v^{m-\beta -1}}{(1+v)^m}dv\right)\\
        & \qquad \qquad \leq C\left(\frac{1}{\sqrt{zx}}+\frac{1}{z}\right)\int_0^\infty \frac{v^{m-\beta -1}}{(1+v)^m}dv
            \leq \frac{C}{z}.
    \end{align*}
    We have obtained that
    \begin{equation}\label{5.10}
        \left(\int_{\Gamma _+(x)}|S_3(y,z,t)|^2\frac{dydt}{t^2}\right)^{1/2}
            \leq \frac{C}{z}.
    \end{equation}
    By combining (\ref{5.8}), (\ref{5.9}) and (\ref{5.10}) we conclude the proof of this lemma.
\end{proof}

\begin{lema}\label{Lem5.4}
    Let $\B$ be a UMD Banach space, $1<p<\infty$ and $\beta, \lambda>0$. Then,
    $$\Big\|t^\beta \partial_t^\beta P_t^{\mathfrak{B}_\lambda}(f)\Big\|_{T^2_p((0,\infty),\B)}
        \leq C\|f\|_{L^p((0,\infty),\B)}, \quad f \in C_c^\infty(0,\infty) \otimes \B.$$
\end{lema}

\begin{proof}
    Assume that $f=\sum_{j=1}^n f_jb_j$, where $f_j\in C_c^\infty(0,\infty )$ and $b_j\in \B$, $j=1,\dots,n$.
    We denote by $f_{\rm o}$ the odd extension of $f$ to $\mathbb{R}$, that is,
    $$ f_{\rm o}(x)=\left\{\begin{array}{ll}
                f(x),&x>0,\\
                -f(-x),&x\leq 0.
                    \end{array}\right.$$

    We have that
    $$P_t(f_{\rm o})(y)
        =\int_{-\infty }^{+\infty }P_t(y-z)f(z)dz=\int_0^\infty (P_t(y-z)-P_t(y+z))f(z)dz,\quad y\in \mathbb{R}\mbox{ and }t>0.$$
    According to Lemmas~\ref{Lem5.1}, \ref{Lem5.2} and \ref{Lem5.3} it follows that
    $$\left\|t^\beta \partial _t^\beta[P_t^{\mathfrak{B}_\lambda }(y,z)-P_t(y-z)]\chi _{\Gamma _+(x)}(y,t)\right\|_{H_+}
        \leq C K(x,z),$$
    where
    $$K(x,z)
        =\left\{\begin{array}{ll}
        \displaystyle \frac{1}{x},&\displaystyle 0<z<\frac{x}{2},\\
        &\\
        \displaystyle \frac{1}{z}\Big(1+\log_+\frac{z}{|x-z|}\Big),&\displaystyle 0< x/2 <z<2x,\\
    &\\
        \displaystyle \frac{1}{z},&0<2x<z.
        \end{array}\right.$$
    For every $x\in (0,\infty )$,
    $$d\mu _x(z)
        =\frac{1}{C_0} \left(1+\log_+\frac{z}{|x-z|}\right) \chi_{(x/2,2x)}(z)\frac{dz}{z},$$
    is a probability measure on $(0,\infty)$, where
    $$C_0= \int_{1/2}^2 \left( 1 + \log_+ \frac{u}{|1-u|} \right) \frac{du}{u}.$$
    Then, by using Jensen inequality and by invoking Hardy inequalities (\cite[p. 20]{Zyg}), we deduce that
    \begin{align}\label{5.12}
        \left\|t^\beta \partial _t^\beta[P_t^{\mathfrak{B}_\lambda }(f)(y) - P_t(f_{\rm o})(y)] \chi _{\Gamma _+(x)}(y,t) \right\|_{L^p((0,\infty ), L^2((0,\infty )^2,\frac{dydt}{t^2}; \B))}
        & \leq C\|K(\|f\|_\B)\|_{L^p(0,\infty )}  \nonumber\\
        & \leq C\|f\|_{L^p((0,\infty ,\B)},
    \end{align}
    where
    $$K(g)(x)
        =\int_0^\infty K(x,z)g(z)dz,\quad x\in (0,\infty ),$$
    for every $g\in L^p(0,\infty )$.

    Note that, in particular, according to Lemma~\ref{Lem2.1},
    $t^\beta \partial _t^\beta P_t(f)\chi _{\Gamma _+(x)}\in L^q((0,\infty),H_+)\otimes \B$, $1<q<\infty$.

    It is clear that if $\{h_j\}_{j=1}^\ell $ is an orthonormal system in $H_+$ and,
    for every $j\in \mathbb{N}$, we define $\widetilde{h_j}$ by
    $$\widetilde{h_j}(y,t)=\left\{\begin{array}{ll}
                h_j(y,t),&y,t>0,\\
                0,&t>0, \ y \leq 0,
                \end{array}\right.$$
    then, $\{\widetilde{h_j}\}_{j=1}^\ell$ is an orthonormal system in $L^2(\mathbb{R}\times (0,\infty ),\frac{dydt}{t^2})$.
    Hence, since $\B$ is UMD we have that
    \begin{align*}
        & \Big\|t^\beta \partial _t^\beta P_t(f_{\rm o})(y)\chi _{\Gamma _+(x)}(y,t) \Big\|_{\gamma (H_+,\B)}
            = \sup \bigg(\E \Big\|\sum_{j=1}^\ell \gamma _j[t^\beta \partial _t^\beta P_t(f_{\rm o})(y)\chi _{\Gamma _+(x)}(y,t)](h_j)\Big\|_\B^2\bigg)^{1/2}\\
        & \qquad  = \sup \bigg(\E \Big\|\sum_{j=1}^\ell\gamma _j[t^\beta \partial _t^\beta P_t(f_{\rm o})(y)](\chi _{\Gamma(x)}(y,t)\widetilde{h_j}) \Big\|_\B^2 \bigg)^{1/2} \\
        & \qquad  \leq  \sup \bigg(\E \Big \|\sum_{j=1}^\ell \gamma _j[t^\beta \partial _t^\beta P_t(f_{\rm o})(y)](\chi _{\Gamma(x)}(y,t)e_j) \Big\|_\B^2 \bigg)^{1/2}\\
        & \qquad  = \Big\|t^\beta \partial _t^\beta P_t(f_{\rm o})(y)\chi _{\Gamma (x)}(y,t)\Big\|_{\gamma \Big(L^2(\mathbb{R}\times (0,\infty),\frac{dydt}{t^2}),\B\Big)},\quad \mbox{ a.e. }x\in (0,\infty ),
    \end{align*}
    where the two first supremum are taken over all orthonormal systems $\{h_j\}_{j=1}^\ell$ in $H_+$ and the last
    one over all orthonormal systems $\{e_j\}_{j=1}^\ell $ in $L^2(\mathbb{R}\times (0,\infty ),\frac{dydt}{t^2})$.

    By Lemma~\ref{Lem2.1},
    \begin{equation}\label{5.13}
        \|t^\beta \partial _t^\beta P_t(f_{\rm o})\|_{T_p^2((0,\infty ),\B)}
            \leq \|t^\beta \partial _t^\beta P_t(f_{\rm o})\|_{T_p^2(\mathbb{R},\B)} \leq C\|f_{\rm o}\|_{L^p(\mathbb{R},\B)}=C\|f\|_{L^p((0,\infty ),\B)}.
    \end{equation}
    Also, as in (\ref{5.12}), since $\gamma  (H_+,\mathbb{C})= H_+$, we get
    \begin{align}\label{5.14}
        &\|t^\beta \partial _t^\beta P_t^{\mathfrak{B}_\lambda }(f)(y)-t^\beta \partial _t^\beta P_t(f_{\rm o})(y)\|_{T_p^2((0,\infty ),\B)}\nonumber\\
        & \qquad =
        \left\|[t^\beta \partial _t^\beta P_t^{\mathfrak{B}_\lambda }(f)(y)-t^\beta \partial _t^\beta P_t(f_{\rm o})(y)]\chi _{\Gamma _+(x)}(y,t)\right\|_{L^p((0,\infty ), \gamma (H_+,\B))}\nonumber\\
        & \qquad \leq\left\|\int_0^\infty \|f(y)\|_\B\|t^\beta \partial _t^\beta [P_t^{\mathfrak{B}_\lambda }(f)(y)- P_t(f_{\rm o})(y)]\chi _{\Gamma _+(x)}(y,t)\|_
        {H_+}dy\right\|_{L^p(0,\infty )}\nonumber\\
        & \qquad \leq C\|K(\|f\|_\B)\|_{L^p(0,\infty )}\leq C\|f\|_{L^p((0,\infty ),\B)}.
    \end{align}
    Combining (\ref{5.13}) and (\ref{5.14}) we conclude the proof of this lemma.
\end{proof}

From Lemma~\ref{Lem5.4} and using the estimate below (Lemma~\ref{Lem5.5}), we can proceed as in the proof of
Lemma~\ref{Lem2.1} to obtain \eqref{5.1}.

\begin{lema}\label{Lem5.5}
    Let $\beta, \lambda >0$. Then,
    $$|t^\beta \partial _t^\beta P_t^{\mathfrak{B}_\lambda }(x,y)|
        \leq C\frac{t^\beta}{(t+|x-y|)^{\beta +1}},\quad t,x,y\in (0,\infty ).$$
\end{lema}

\begin{proof}
    Let $m \in \N$ such that $m-1 \leq \beta < m$ and $k\in \mathbb{N}$, verifying $0\leq k\leq (m+1)/2$.
    According to the formula (\ref{5.3}), it is enough to estimate the following expression,
    \begin{align*}
        &\left|\frac{(xy)^\lambda }{t^{2\lambda +1}}\int_0^\pi (\sin \theta )^{2\lambda -1}
            \varphi ^{\lambda ,k}\left(\frac{\sqrt{(x-y)^2+2xy(1-\cos \theta )}}{t}\right)d\theta \right|\\
        &\qquad \leq \frac{(xy)^\lambda }{t^{2\lambda +1}}\int_0^\pi (\sin \theta )^{2\lambda -1}\int_0^\infty \frac{(1+v)^{m+1-2k}v^{m-\beta -1}}{((1+v)^2+\frac{(x-y)^2+2xy(1-\cos \theta )}{t^2})^{\lambda +m-k+1}}dvd\theta\\
        &\qquad \leq (xy)^\lambda t^{2m-2k+1}\int_0^\pi (\sin \theta )^{2\lambda -1}\int_0^\infty \frac{(1+v)^{m+1-2k}v^{m-\beta -1}
        }{(t^2(1+v)^2+(x-y)^2+2xy(1-\cos \theta ))^{\lambda +m-k+1}}dvd\theta\\
        &\qquad \leq C(xy)^\lambda t^{2m-2k+1}\int_0^{\pi/2}(\sin \theta )^{2\lambda -1}\int_0^\infty \frac{(1+v)^{m+1-2k}v^{m-\beta -1}}{(t^2(1+v)^2+(x-y)^2+2xy(1-\cos \theta ))^{\lambda +m-k+1}}dvd\theta\\
        &\qquad \leq C(xy)^\lambda t^{2m-2k+1}\int_0^\infty (1+v)^{m+1-2k}v^{m-\beta -1}\int_0^{\pi/2}\frac{\theta ^{2\lambda -1}}{(t^2(1+v)^2+(x-y)^2+xy\theta ^2)^{\lambda +m-k+1}}d\theta dv\\
        &\qquad \leq Ct^{2m-2k+1}\int_0^\infty \frac{(1+v)^{m+1-2k}v^{m-\beta -1}}{(t(1+v)+|x-y|)^{2m-2k+2}}dv\\
        &\qquad \leq Ct^m\int_0^\infty \frac{v^{m-\beta -1}}{(tv+t+|x-y|)^{m+1}}dv\leq C\frac{t^\beta}{(t+|x-y|)^{\beta +1}},\quad t,x,y\in (0,\infty ).
    \end{align*}
\end{proof}

\subsection{}\label{subsec:Bess2}
Now we are going to show that
$$\|f\|_{L^p((0,\infty ), \B)}
    \leq C\|t^\beta \partial _t^\beta P_t^{\mathfrak{B}_\lambda }(f)\|_{T_p^2((0,\infty ),\B)},\quad f\in L^p((0,\infty ), \B).$$

If $f\in C_c^\infty (0,\infty )$ then, according to \cite[Lemma 3.1]{BCR3} we have that
$$h_\lambda (t^\beta \partial _t^\beta P_t^{\mathfrak{B}_\lambda }(f))(x)
    =e^{i\pi \beta }(tx)^\beta e^{-tx}h_\lambda (f)(x),\quad t,x\in (0,\infty ).$$
Suppose that $f,g\in C_c^\infty (0,\infty )$. We denote by
$J_t(y)=\{x\in (0,\infty ):|x-y|<t\}$, $t, y \in  (0,\infty )$.
Plancherel equality for Hankel transform leads to
\begin{align}\label{5.15}
    & \int_0^\infty \int_{\Gamma _+(x)}t^\beta \partial _t^\beta P_t^{\mathfrak{B}_\lambda }(f)(y)t^\beta \partial _t^\beta P_t^{\mathfrak{B}_\lambda }(g)(y)\frac{dydt}{t|J_t(y)|}dx \nonumber\\
    & \qquad = \int_0^\infty \int_0^\infty
        h_\lambda [e^{i\pi \beta} (tz)^\beta e^{-tz}h_\lambda(f)](y)
        h_\lambda [e^{i\pi \beta} (tz)^\beta e^{-tz}h_\lambda(g)](y)\int_{J_t(y)}dx\frac{dy}{t|J_t(y)|}dt\nonumber\\
    &\qquad =\int_0^\infty \int_0^\infty e^{2i\pi \beta }(ty)^{2\beta }e^{-2ty}h_\lambda (f)(y)h_\lambda (g)(y)\frac{dydt}{t}\nonumber\\
    &\qquad =e^{2i\pi \beta }\int_0^\infty h_\lambda (f)(y)h_\lambda (g)(y)y^{2\beta}\int_0^\infty e^{-2ty}t^{2\beta -1}dtdy\nonumber\\
    &\qquad = \frac{e^{2i\pi \beta }\Gamma (2\beta)}{2^{2\beta }}\int_0^\infty h_\lambda (f)(y)h_\lambda (g)(y)dy
        =\frac{e^{2i\pi \beta }\Gamma (2\beta)}{2^{2\beta }}\int_0^\infty f(x)g(x)dx.
\end{align}
Suppose now that $f\in C_c^\infty (0,\infty )\otimes \B$ and $g\in C_c^\infty (0,\infty )\otimes \B^*$. From (\ref{5.15}) it follows that
$$\int_0^\infty \langle g(x), f(x)\rangle_{\B^*,\B}dx
    =\frac{e^{-2i\pi \beta }2^{2\beta }}{\Gamma (2\beta)}\int_0^\infty \int_{\Gamma _+(x)}\langle t^\beta \partial _t^\beta P_t^{\mathfrak{B}_\lambda }(g)(y), t^\beta \partial _t^\beta P_t^{\mathfrak{B}_\lambda }(f)(y)\rangle_{\B^*,\B}\frac{dydt}{t|J_t(y)|}dx.$$
Then, since $|J_t(y)|\geq t$, for every $t, y \in  (0,\infty )$, according to \cite[Proposition 2.2]{HW} we get
\begin{align*}
    & \left|\int_0^\infty \langle g(x), f(x)\rangle _{\B^*,\B}dx\right|
        \leq C\int_0^\infty \int_{\Gamma _+(x)}|\langle t^\beta \partial _t^\beta P_t^{\mathfrak{B}_\lambda }(g)(y), t^\beta \partial _t^\beta P_t^{\mathfrak{B}_\lambda }(f)(y)\rangle_{\B^*,\B}|\frac{dydt}{t^2}dx\\
    & \qquad \leq C\int_0^\infty   \|t^\beta \partial _t^\beta P_t^{\mathfrak{B}_\lambda }(f)(y)\chi _{\Gamma_+(x)}(y,t)\|_{\gamma (H_+,\B)}\|t^\beta \partial _t^\beta P_t^{\mathfrak{B}_\lambda }(g)(y)\chi _{\Gamma_+(x)}(y,t)\|_{\gamma (H_+,\B^*)}                               dx\\
    & \qquad \leq C\|t^\beta \partial _t^\beta P_t^{\mathfrak{B}_\lambda }(f)\|_{T_p^2((0,\infty ),\B)}\|t^\beta \partial _t^\beta P_t^{\mathfrak{B}_\lambda }(g)\|_{T_{p'}^2((0,\infty ),\B)},
\end{align*}
where
$p'=p/(p-1)$.

Since $\B^*$ is UMD, by using (\ref{5.1}) where $p$ is replaced by $p'$, we obtain
$$\left|\int_0^\infty \langle g(x),f(x)\rangle _{\B^*,\B}dx\right|
    \leq C\|t^\beta \partial _t^\beta P_t^{\mathfrak{B}_\lambda }(f)\|_{T_p^2((0,\infty ),\B)})||g||_{L^{p'}((0,\infty ),\B^*)}.$$
From \cite[Lemma 2.3]{GLY} and by taking into account that $C_c^\infty (0,\infty )\otimes \B^*$ is dense in $L^{p'}((0,\infty ), \B^*)$ we get
$$\|f\|_{L^p((0,\infty ),\B)}
    \leq C\|t^\beta \partial _t^\beta P_t^{\mathfrak{B}_\lambda }(f)\|_{T_p^2((0,\infty ),\B)},\quad f\in C_c^\infty (0,\infty )\otimes \B.$$
By using again (\ref{5.1}) and the fact that $C_c^\infty (0,\infty )\otimes \B$ is dense in $L^p((0,\infty ),\B)$ we conclude that
$$\|f\|_{L^p((0,\infty ),\B)}
    \leq C\|t^\beta \partial _t^\beta P_t^{\mathfrak{B}_\lambda }(f)\|_{T_p^2((0,\infty ),\B)},\quad f\in L^p((0,\infty ,\B).$$

\section{Proof of Theorem \ref{teo 1.4}}\label{sec:Laguerre}

\subsection{}
We are going to show that
\begin{equation}\label{6.1}
    \|t^\beta \partial _t^\beta P_t^{\mathcal{L}_\alpha}(f)\|_{T_p^2((0,\infty ),\B)}
        \leq C  \|f\|_{L^p((0,\infty ),\B)}, \quad f\in L^p((0,\infty ,\B).
\end{equation}

We have that
$$L_\alpha =-\frac{d^2}{dx^2}+V(x), \quad x\in (0,\infty ),$$
where $V(x)=x^2 + (\alpha ^2-1/4)/x^2$.
Then, $\partial _t P_t^{\mathcal{L}_\alpha}(1)\not=0$ and \cite[Theorem 4.8]{HNP} can not be applied to prove (\ref{6.1}).
The strategy will be the same as in previous sections: comparing with an operator whose boundedness property is already known.
This time we are going to relate the operator $t^\beta \partial _t^\beta P_t^{\mathcal{L}_\alpha}$ with
$t^\beta \partial _t^\beta P_t^{\mathcal{H}}$, studied in Section~\ref{sec:Hermite}.

Recall \eqref{eq:PoissLag} for the definition of the Poisson Laguerre semigroup.
It is given via subordination with respect to the heat Laguerre semigroup $W_t^{\mathcal{L}_\alpha}$,
in which it is involved the modified Bessel function $I_\alpha$ (see \eqref{eq:heatLag}).
We will use the following properties of $I_\alpha$, that can be encountered in \cite[Ch. 5]{Leb},
\begin{eqnarray}
    I_\alpha (z)\sim \frac{z^\alpha }{2^\alpha \Gamma (\alpha +1)},\quad \mbox{ as }z\rightarrow 0^+,&&\label{6.2}\\
    \sqrt{z}I_\alpha (z)=\frac{e^z}{\sqrt{2\pi }}\left(1+O\Big(\frac{1}{z}\Big)\right),\quad z\in (0,\infty).&&\label{6.3}
\end{eqnarray}
From (\ref{6.2}) and (\ref{6.3}) we easily deduce that
\begin{equation}\label{6.4}
    0\leq W_t^{\mathcal{L}_\alpha}(x,y)
        \leq C \frac{e^{-c(x-y)^2/t}}{\sqrt{t}},\quad t,x,y\in (0,\infty ).
\end{equation}

Let $f\in C_c^\infty (0,\infty )\otimes \B$. We define the measurable function $f_{0}$ by
$$f_0(x)=\left\{
    \begin{array}{ll}
        f(x),&x>0\\
        0,&x\leq0.
    \end{array}\right.$$
By using \eqref{3.3}, \eqref{4.2}, \eqref{6.4} and H\"older's inequality, it is not difficult to justify that
$$t^\beta \partial _t^\beta P_t^{\mathcal{L}_\alpha}(f)(y)
    =\int_0^\infty t^\beta \partial _t^\beta P_t^{\mathcal{L}_\alpha}(y,z)f(z)dz,\quad t, y \in  (0,\infty ),$$
and
$$t^\beta \partial _t^\beta P_t^\mathcal{H}(f_0)(y)
    =\int_0^\infty t^\beta \partial _t^\beta P_t^\mathcal{H}(y,z)f(z)dz,\quad t, y \in  (0,\infty ).$$
Here, $P_t^\mathcal{H}(x,y)$ denotes the Poisson kernel associated with the Hermite operator on $\mathbb{R}$.

For every $x\in (0,\infty )$, we write the following decomposition
\begin{align*}
    & t^\beta \partial _t^\beta [P_t^{\mathcal{L}_\alpha }(f)(y) - P_t^\mathcal{H}(f_0)(y)] \chi_{\Gamma_+(x)}(y,t)\\
    & \qquad =  \chi_{\Gamma_+(x)}(y,t) \left(\int_0^{x/2}+\int_{2x}^\infty +\int_{x/2}^{2x}\right)
        t^\beta \partial _t^\beta [P_t^{\mathcal{L}_\alpha }(y,z) - P_t^\mathcal{H}(y,z)] f(z)dz\\
    & \qquad = K_{glob}^1(f)(x,y,t) + K_{glob}^2(f)(x,y,t) + K_{{\rm loc}}(f)(x,y,t), \quad t, y \in  (0,\infty ).
\end{align*}

Next, we analyze the boundedness properties of each of the above operators.

\begin{lema}\label{Lem6.1}
    Let $\B$ be a Banach space and $1<p<\infty$. Then, for each $j=1,2$,
    $$\| K^j_{glob}(f) \|_{L^p((0,\infty),\gamma(H_+,\B))}
        \leq C \|f\|_{L^p((0,\infty),\B)}, \quad f \in C_c^\infty(0,\infty) \otimes \B.$$
\end{lema}

\begin{proof}
  Let $f\in  C_c^\infty(0,\infty) \otimes \B$. By using (\ref{3.3}), (\ref{4.2}) and (\ref{6.4}) we deduce that
    \begin{align}\label{6.55}
        & |t^\beta \partial _t^\beta [P_t^{\mathcal{L}_\alpha }(y,z)- P_t^\mathcal{H}(y,z)]|
            \leq C\int_0^\infty \frac{|t^\beta \partial _t^\beta [te^{-t^2/4u}]|}{u^{3/2}}|W_u^{\mathcal{L}_\alpha }(y,z)-W_u^\mathcal{H}(y,z)|du\nonumber\\
        & \qquad \leq C t^\beta \int_0^\infty \frac{e^{-c(t^2+(y-z)^2)/u}}{u^{(\beta +3)/2}}du\leq C\frac{t^\beta }{(t+|y-z|)^{\beta +1}},\quad  t,y,z\in (0,\infty ).
    \end{align}
    Hence we obtain
    \begin{align}\label{32.1}
        & \Big\|t^\beta \partial _t^\beta [P_t^{\mathcal{L}_\alpha }(y,z) -  P_t^\mathcal{H}(y,z))]\chi_{\Gamma _+(x)}(y,t)\Big\|_{H_+}
            \leq C\left(\int_0^\infty \int_{|x-y|<t}\frac{t^{2\beta -2}}{(t+|y-z|)^{2\beta +2}}dydt\right)^{1/2} \nonumber \\
        & \qquad \leq C\left(\int_0^\infty  \int_{|x-y|<t}\frac{t^{2\beta -2}}{(t+|x-y|+|y-z|)^{2\beta +2}}dydt\right)^{1/2} \nonumber \\
        & \qquad \leq C\left(\int_0^\infty \int_{|x-y|<t}\frac{t^{2\beta -2}}{(t+|x-z|)^{2\beta +2}}dydt\right)^{1/2}\leq C\left(\int_0^\infty
            \frac{t^{2\beta -1}}{(t+|x-z|)^{2\beta +2}}dt\right)^{1/2} \nonumber \\
        & \qquad \leq \frac{C}{|x-z|}\leq C\left\{\begin{array}{ll}
                            \displaystyle \frac{1}{z},& 0<2x<z, \\
                            &\\
                            \displaystyle \frac{1}{x},& \displaystyle 0<z<\frac{x}{2}.
                            \end{array}  \right.
    \end{align}
    Then, since $\gamma (H_+,\mathbb{C})=H_+$, we get
    \begin{align*}
        &\|K_{glob}^1(f)(x,y,t)\|_{\gamma (H_+,\B)} + \|K_{glob}^2(f)(x,y,t)\|_{\gamma (H_+,\B)} \\
        & \qquad \leq C \Big( \int_0^{x/2} + \int_{2x}^\infty \Big)\Big\|t^\beta \partial _t^\beta [P_t^{\mathcal{L}_\alpha }(y,z)- P_t^\mathcal{H}(y,z)]\chi_{\Gamma _+(x)}(y,t)\Big\|_{H_+}\|f(z)\|_Bdz\\
        & \qquad \leq C \Big( \frac{1}{x}\int_0^{x/2}\|f(z)\|_Bdz + \int_{2x}^\infty \frac{\|f(z)\|_\B}{z}dz \Big) ,\quad x\in (0,\infty ).
    \end{align*}
    By using Hardy inequalities (\cite[p. 20]{Zyg}) we conclude that
    \begin{equation*}\label{6.6}
        \|K_{glob}^j(f)\|_{L^p((0,\infty),\gamma (H_+,\B))}
            \leq C\|f\|_{L^p((0,\infty ),\B)},\quad j=1,2.
    \end{equation*}
\end{proof}

\begin{lema}\label{Lem6.2}
    Let $\B$ be a Banach space and $1<p<\infty$. Then,
    $$\| K_{loc}(f) \|_{L^p((0,\infty),\gamma(H_+,\B))}
        \leq C \|f\|_{L^p((0,\infty),\B)}, \quad f \in C_c^\infty(0,\infty) \otimes \B.$$
\end{lema}

\begin{proof}
      Let $f\in  C_c^\infty(0,\infty) \otimes \B$. To simplify notation, we call
    $$\xi
        =\xi(u,y,z)=\frac{2yze^{-2u}}{1-e^{-4u}}, \quad u,y,z \in (0,\infty).$$
    We  make the following decomposition
    \begin{align*}
        & t^\beta \partial _t^\beta [P_t^{\mathcal{L}_\alpha }(y,z)- P_t^\mathcal{H}(y,z)]\\
        & \qquad = \frac{t^\beta }{2\sqrt{\pi }}\left(\int_{\{u\in (0,\infty ) \, : \, \xi \leq 1\}}+\int_{\{u\in (0,\infty ) \, : \, \xi \geq 1\}}\right)
            \frac{\partial _t^\beta [te^{-t^2/4u}]}{u^{3/2}}[W_u^{\mathcal{L}_\alpha }(y,z)-W_u^\mathcal{H}(y,z)]du\\
        & \qquad = I_1(y,z,t)+I_2(y,z,t),\quad t,y,z\in (0,\infty ).
    \end{align*}

    According to (\ref{3.3}) and (\ref{6.2}) we get
    \begin{align*}
        &|I_1(y,z,t)| \\
        & \qquad \leq  Ct^\beta\int_{\{u\in (0,\infty ) \, : \, \xi \leq 1\}} \frac{e^{-t^2/8u}}{u^{(\beta +2)/2}}
                \left(\frac{e^{-2u}}{1-e^{-4u}}\right)^{1/2}\exp\left(-\frac{1}{2}(y^2+z^2)\frac{1+e^{-4u}}{1-e^{-4u}}\right)
                \left[\xi^{\alpha +1/2} + e^\xi \right]du\\
        & \qquad \leq  Ct^\beta \int_0^\infty \frac{e^{-c(t^2+y^2+z^2)/u}}{u^{(\beta +3)/2}}du
            \leq  C\frac{t^\beta }{(t+y+z)^{\beta +1}},\quad t,z,y\in (0,\infty ).
    \end{align*}

    By proceeding as in \eqref{32.1} we obtain
    \begin{equation}\label{6.7}
        \|I_1(y,z,t)\chi _{\Gamma _+(x)}(y,t)\|_{H_+}
            \leq \frac{C}{z+x}\leq \frac{C}{x},\quad 0< x/2 <z<2x.
    \end{equation}

    On the other hand, (\ref{3.3}) and (\ref{6.3}) lead to
    \begin{align*}
        |I_2(y,z,t)|
            &\leq Ct^\beta \int_{\{u\in (0,\infty ) \, : \, \xi \geq 1\}}
                \frac{e^{-t^2/8u}}{u^{(\beta +2)/2}} W_u^\mathcal{H}(y,z) \frac{du}{\xi} \\
            & \leq Ct^\beta\int_{\{u\in (0,\infty ) \, : \, \xi \geq 1\}}
                \frac{e^{-t^2/8u}}{u^{(\beta +2)/2}} \frac{e^{-u}e^{-c|y-z|^2/u}}{u^{1/2}} \frac{du}{\xi^{1/3}} \\
            &\leq C\frac{t^\beta}{(yz)^{1/3}}\int_0^\infty
                \frac{e^{-c(t^2+|y-z|^2)/u}}{u^{(3\beta +7)/6}}du\leq C\frac{t^\beta }{(yz)^{1/3}(t+|y-z|)^{\beta +1/3}},\quad t,z,y\in (0,\infty ).
    \end{align*}
    Then, by choosing $0<\varepsilon <\min \{2/3, 2\beta\}$ we can write
    \begin{align}\label{6.8}
        &\|I_2(y,z,t)\chi _{\Gamma _+(x)}(y,t)\|_{H_+}
            \leq C\left(\int_{\Gamma _+(x)}\frac{t^{2\beta -2}}{(yz)^{2/3}(t+|y-z|)^{2\beta +2/3}}dtdy\right)^{1/2}\nonumber\\
        & \qquad \leq C\left( \left\{\int_0^{z/4} + \int_{z/4}^\infty \right\}\int_{|x-y|}^\infty \frac{t^{2\beta -2}}{(yz)^{2/3}(t+|y-z|)^{2\beta +2/3}}dtdy \right)^{1/2}\nonumber\\
        & \qquad \leq C\left(\int_0^{z/4}\int_{x/2}^\infty \frac{dtdy}{(yz)^{2/3}t^{8/3}}+\int_{z/4}^\infty \frac{1}{(yz)^{2/3}|x-z|^{\varepsilon + 2/3}}\int_{|x-y|}^\infty \frac{dt}{t^{2-\varepsilon}}dy\right)^{1/2}\nonumber\\
        & \qquad \leq C\left(\int_0^{z/4} \frac{dy}{(yz)^{2/3}x^{5/3}}+\frac{1}{z^{2/3}|x-z|^{\varepsilon + 2/3}}\int_{x/8}^\infty \frac{dy}{y^{2/3}|x-y|^{1-\varepsilon}}\right)^{1/2}\nonumber\\
        & \qquad \leq C\left(\frac{1}{z^{1/3}x^{5/3}}+\frac{1}{z^{2/3}|x-z|^{\varepsilon + 2/3}x^{-\varepsilon + 2/3}}\int_{1/8}^\infty \frac{du}{u^{2/3}|1-u|^{1-\varepsilon}}\right)^{1/2}\nonumber\\
        & \qquad \leq C\left(\frac{1}{x^2}+\frac{x^{\varepsilon + 2/3}}{x^2|x-z|^{\varepsilon + 2/3}}\right)^{1/2}
            \leq\frac{C}{x}\left(1+\left(\frac{x}{|x-z|}\right)^{(3 \varepsilon + 2)/6}\right), \quad 0< x/2 <z<2x.
    \end{align}
    By combining (\ref{6.7}) and (\ref{6.8}) we get
    $$\Big\|t^\beta \partial _t^\beta [P_t^{\mathcal{L}_\alpha }(y,z)- P_t^\mathcal{H}(y,z)]\chi_{\Gamma _+(x)}(y,t)\Big\|_{H_+}
        \leq \frac{C}{x}\left(1+\left(\frac{x}{|x-z|}\right)^{(3 \varepsilon + 2)/6}\right), \quad 0< x/2 <z<2x.$$
    For every $x\in (0,\infty )$,
    $$d_{\mu _x}(z)
        =\frac{1}{C_0} \left(1+\Big(\frac{x}{|x-z|}\Big)^{(3 \varepsilon + 2)/6}\right)\chi _{(x/2,2x)}(z) \frac{dz}{z}$$
    is a probability measure on $(0,\infty )$
    when
    $$C_0=\int_{1/2}^2\left(1+\frac{1}{|1-u|^{(3 \varepsilon + 2)/6}}\right)du.$$
    Then, since $\gamma (H_+,\mathbb{C})=H_+$, Jensen's inequality leads to
    \begin{align*}
        \|K_{\rm loc}(f)(x,\cdot,\cdot)\|^p_{\gamma (H_+,\B)}
            &\leq C \left(\int_{x/2}^{2x}\|f(z)\|_\B
                \Big\|t^\beta \partial _t^\beta [P_t^{\mathcal{L}_\alpha }(y,z)- P_t^\mathcal{H}(y,z)]\chi _{\Gamma_+(x)}(y,t)\Big\|_{H_+}dz\right)^p\\
        &  \leq C\left(\frac{1}{x}\int_{x/2}^{2x}\|f(z)\|_\B\left(1+\Big(\frac{x}{|x-z|}\Big)^{(3 \varepsilon + 2)/6}\right)dz\right)^p\\
        &  \leq C\int_{x/2}^{2x}\|f(z)\|_\B^p\frac{1}{x}\left(1+\Big(\frac{x}{|x-z|}\Big)^{(3 \varepsilon + 2)/6}\right)dz,\quad x\in (0,\infty ).
    \end{align*}
    Hence,
    \begin{align*}\label{6.9}
        & \|K_{\rm loc}(f)\|_{L^p((0,\infty ),\gamma (H_+,\B))}
            \leq C\left(\int_0^\infty\int_{x/2}^{2x}\|f(z)\|^p_\B\frac{1}{x}\left(1+\Big(\frac{x}{|x-z|}\Big)^{(3 \varepsilon + 2)/6}\right)dzdx\right)^{1/p}\nonumber\\
        &  \qquad \leq C\left(\int_0^\infty\|f(z)\|^p_\B\int_{z/2}^{2z}\frac{1}{x}\left(1+\Big(\frac{x}{|x-z|}\Big)^{(3 \varepsilon + 2)/6}\right)dxdz\right)^{1/p}
            \leq C\|f(z)\|_{L^p((0,\infty ),\B)}.
    \end{align*}
\end{proof}

Putting together Lemmas~\ref{Lem6.1} and \ref{Lem6.2} we conclude that
$$\|t^\beta \partial _t^\beta P_t^{\mathcal{L}_\alpha }(f)-t^\beta \partial _t^\beta P_t^\mathcal{H}(f_0) \|_{T_p^2((0,\infty ),\B)}
    \leq \|f\|_{L^p((0,\infty ),\B)}, \quad f \in C_c^\infty(0,\infty) \otimes \B.$$
Moreover, according to Theorem~\ref{Theorem 1.2}, $(ii)$, since $\B$ is a UMD Banach space, we have that
\begin{align*}
    \|t^\beta \partial _t^\beta P_t^\mathcal{H}(f_0) \|_{T_p^2((0,\infty ),\B)}
        & \leq \|t^\beta \partial _t^\beta P_t^\mathcal{H}(f_0)\|_{T_p^2(\R,\B)} \\
        & \leq C\|f\|_{L^p((0,\infty ), \B)},\quad f \in C_c^\infty(0,\infty) \otimes \B.
\end{align*}
Hence,
$$\|t^\beta \partial _t^\beta P_t^{\mathcal{L}_\alpha }(f)\|_{T_p^2((0,\infty ),\B)}
    \leq \|f\|_{L^p((0,\infty ),\B)}, \quad f \in C_c^\infty(0,\infty) \otimes \B.$$
By proceeding as in (\ref{6.55}) we get
$$|t^\beta \partial _t^\beta P_t^{\mathcal{L}_\alpha }(y,z)|
    \leq C\frac{t^\beta }{(t+|y-z|)^{\beta +1}},\quad t,y,z\in (0,\infty ).$$
Finally,  as in the proof of Lemma~\ref{Lem2.1}, we can deduce that
$$\|t^\beta \partial _t^\beta P_t^{\mathcal{L}_\alpha }(f) \|_{T_p^2((0,\infty ),\B)}
    \leq \|f\|_{L^p((0,\infty ),\B)},\quad f\in L^p((0,\infty ),\B).$$

\subsection{}
In this paragraph we establish that
$$\|f\|_{L^p((0,\infty ), \B)}
    \leq C\|t^\beta \partial _t^\beta P_t^{\mathcal{L}_\alpha }(f)\|_{T_p^2((0,\infty ), \B)}, \quad f \in L^p((0,\infty),\B).$$
As in Section~\ref{subsec:Bess2}, by duality and density arguments, it is enough to have the following identity.

\begin{lema}\label{Lem6.3}
    Let $\alpha,\beta >0$ and $f,g\in C_c^\infty (0,\infty )$. Then,
    \begin{equation*}\label{6.10}
        \int_0^\infty \int_{\Gamma _+(x)}t^\beta \partial _t^\beta P_t^{\mathcal{L}_\alpha }(f)(y)t^\beta \partial _t^\beta P_t^{\mathcal{L}_\alpha }(g)(y)\frac{dydt}{t|J_t(y)|}dx
            =\frac{e^{2\pi i\beta} \Gamma (2\beta)}{2^{2\beta}}\int_0^\infty f(x)g(x)dx,
    \end{equation*}
    where $J_t(y)=\{x\in (0,\infty ): |x-y|<t\}$.
\end{lema}

\begin{proof}
    We can write
    \begin{equation}\label{6.11}
        P_t^{\mathcal{L}_\alpha }(f)(x)
            =\sum_{k=0}^\infty e^{-t\sqrt{2k+\alpha +1}}c_k^\alpha (f)\varphi _k^\alpha (x),\quad  t,x\in (0,\infty ).
    \end{equation}
    According to  \cite[p. 1124]{MW} there exists $C>0$ such that $|\varphi _k^\alpha (x)|\leq C$, $x\in (0,\infty )$. Moreover, for every $m\in \mathbb{N}$ there exists $C>0$ such that $|c_k^\alpha (f)|\leq C(1+k)^{-m}$, $k\in \mathbb{N}$. Hence, the series in (\ref{6.11}) converges in $L^2(0,\infty )$ and uniformly in $(t,x)\in (0,\infty )\times (0,\infty )$. Also, for every $m\in \mathbb{N}$, we get
    $$\partial _t^mP_t^{\mathcal{L}_\alpha }(f)(x)
        =\sum_{k=0}^\infty (-1)^m(2k+\alpha +1)^{m/2}e^{-t\sqrt{2k+\alpha +1}}c_k^\alpha (f)\varphi _k^\alpha (x),\quad x,t\in (0,\infty ).$$
    If $m\in \mathbb{N}$ is such that $m-1\leq \beta <m$, then
    \begin{align*}
        \partial _t^\beta P_t^{\mathcal{L}_\alpha }(f)(x)
            &=\frac{e^{-i\pi (m-\beta)}}{\Gamma (m-\beta)}\int_0^\infty \partial _t^mP_{t+s}^{\mathcal{L}_\alpha }(f)(x)s^ {m-\beta -1}ds\\
            &=\frac{e^{-i\pi (m-\beta)}}{\Gamma (m-\beta)}\int_0^\infty s^ {m-\beta -1}\sum_{k=0}^\infty (-1)^m(2k+\alpha +1)^{m/2}e^{-(t+s)\sqrt{2k+\alpha +1}}c_k^\alpha (f)\varphi _k^\alpha (x)ds\\
            &=\frac{e^{-i\pi (m-\beta)}}{\Gamma (m-\beta)}\sum_{k=0}^\infty (-1)^m(2k+\alpha +1)^{m/2}e^{-t\sqrt{2k+\alpha +1}}c_k^\alpha (f)\varphi _k^\alpha (x)\int_0^\infty s^ {m-\beta -1}
                e^{-s\sqrt{2k+\alpha +1}}ds\\
            &=e^{i\pi \beta}\sum_{k=0}^\infty (2k+\alpha +1)^{\beta/2}e^{-t\sqrt{2k+\alpha +1}}c_k^\alpha (f)\varphi _k^\alpha (x),\quad t,x\in (0,\infty ).
    \end{align*}
    The interchange between the serie and the integral is justified because
    $(c_k^\alpha(f))_{k\in \mathbb{N}}$ is rapidly decreasing as $k\rightarrow \infty$.
    We have that
    \begin{align*}
        & \int_0^\infty \int_{\Gamma _+(x)}t^\beta \partial _t^\beta P_t^{\mathcal{L}_\alpha }(f)(y)t^\beta \partial _t^\beta P_t^{\mathcal{L}_\alpha }(g)(y)\frac{dydt}{t|J_t(y)|}dx&&\\
        & \qquad = \int_0^\infty \int_0^\infty t^\beta \partial _t^\beta P_t^{\mathcal{L}_\alpha }(f)(y)t^\beta \partial _t^\beta P_t^{\mathcal{L}_\alpha }(g)(y)\int_{J_t(y)}dx\frac{dydt}{t|J_t(y)|}\\
        & \qquad = \int_0^\infty \int_0^\infty t^\beta \partial _t^\beta P_t^{\mathcal{L}_\alpha }(f)(y)t^\beta \partial _t^\beta P_t^{\mathcal{L}_\alpha }(g)(y)\frac{dydt}{t}\\
        & \qquad = e^{2i\pi \beta}\int_0^\infty \int_0^\infty \left[\sum_{k=0}^\infty (t\sqrt{2k+\alpha +1})^\beta e^{-t\sqrt{2k+\alpha +1}}c_k^\alpha (f)\varphi _k^\alpha (y)\right]\\
        & \qquad \qquad \times\left[\sum_{m=0}^\infty (t\sqrt{2m+\alpha +1})^\beta e^{-t\sqrt{2m+\alpha +1}}c_m^\alpha (g)\varphi _m^\alpha (y)\right]\frac{dydt}{t}\\
        & \qquad = e^{2i\pi \beta}\sum_{k=0}^\infty c_k^\alpha (f)c_k^\alpha (g)\int_0^\infty (t\sqrt{2k+\alpha +1})^{2\beta} e^{-2t\sqrt{2k+\alpha +1}}\frac{dt}{t}\\
        & \qquad = e^{2i\pi \beta}\frac{\Gamma (2\beta)}{2^{2\beta}}\sum_{k=0}^\infty c_k^\alpha (f)c_k^\alpha (g)=e^{2i\pi \beta}\frac{\Gamma (2\beta)}{2^{2\beta}}\int_0^\infty f(x)g(x)dx.
    \end{align*}
\end{proof}


\def\cprime{$'$}

\end{document}